\documentclass[11pt]{article}
\usepackage{a4wide}
\usepackage{amsmath}
\usepackage{amssymb}
\usepackage{amsthm}
\usepackage{xcolor}

\newtheorem{theorem}{Theorem}
\newtheorem{problem}[theorem]{Problem}
\newtheorem{conjecture}[theorem]{Conjecture}
\newtheorem{proposition}[theorem]{Proposition}

\newtheorem{definition}[theorem]{Definition}
\newtheorem{lemma}[theorem]{Lemma}
\newtheorem{corollary}[theorem]{Corollary}
\newtheorem{remark}[theorem]{Remark}

\newcommand{\Fqd}{\mathbb{F}_q^d}

\usepackage{parskip}
\usepackage{hyperref}
\usepackage{parskip}
\usepackage{setspace}
\def\F{\mathbb{F}}

\begin{document}

\title{Additive structures imply more distances in $\mathbb{F}_q^d$}
\author{Daewoong Cheong\and Gennian Ge\and Doowon Koh\and Thang Pham\and Dung The Tran\and Tao Zhang}
\date{}
\maketitle
\begin{abstract}
For a set $E \subseteq \mathbb{F}_q^d$, the distance set is defined as $\Delta(E) := \{\|\mathbf{x} - \mathbf{y}\| : \mathbf{x}, \mathbf{y} \in E\}$, where $\|\cdot\|$ denotes the standard quadratic form. We investigate the Erd\H{o}s--Falconer distance problem within the flexible class of $(u, s)$--Salem sets introduced by Jonathan M. Fraser, with emphasis on the even case $u = 4$. By exploiting the exact identity between $\|\widehat{E}\|_4$ and the fourth additive energy $\Lambda_4(E)$, we prove that quantitative gains in $\Lambda_4(E)$ force the existence of many distances. In particular, for a $(4, s)$--Salem set $E\subset \mathbb{F}_q^d$ with $d \geq 2$, if
\[
|E|\gg  q^{\min\left\{\frac{d+2}{4s+1}, \frac{d+4}{8s}\right\}},
\]
then $E$ determines a positive proportion of all distances. This strictly improves Fraser's threshold of $q^{\frac{d}{4s}}$ and, in certain parameter ranges,  the Iosevich-Rudnev bound of $q^{\frac{d+1}{2}}$. As applications, we obtain improved thresholds for multiplicative subgroups and sets on arbitrary varieties, and establish a sharp incidence bound for Salem sets that is of independent interest in incidence geometry. Moreover, our methods give sharp lower bounds for the number of distinct distances determined by two different sets. We also propose a unified conjecture for $(4, s)$--Salem sets that reconciles known bounds and clarifies the Erd\H{o}s--Falconer distance problem on odd-dimensional spheres: in odd dimensions $d \geq 3$, the often-cited $q^{\frac{d-1}{2}}$ threshold does not hold in general without additional structural assumptions. This provides a clear picture of the spherical distance conjecture. Our results reveal a precise interplay between Fourier decay, additive combinatorics, and incidence geometry.
\end{abstract}



\tableofcontents

\section{Introduction}
Let $\F_q$ be a finite field with $q$ elements and write $\F_q^d$ for the $d$--dimensional vector space equipped with the quadratic form
\[
\|\mathbf{x}\|:=x_1^2+\cdots+x_d^2.
\]
For $E\subseteq\F_q^d$, the distance set of $E$ is defined by
\[
\Delta(E):=\{\|\mathbf{x}-\mathbf{y}\|\;:\;\mathbf{x},\mathbf{y}\in E\}\subseteq\F_q.
\]

The following conjecture is known as the Erd\H{o}s--Falconer distance conjecture in the finite field setting. This conjecture represents the discrete version of the celebrated Falconer distance conjecture \cite{falconer} in geometric measure theory. Recent progress and applications concerning entropy inequalities and discretized sum-product estimates can be found in \cite{Du3, Du4, alex-fal, entropy}.

\begin{conjecture}[Erd\H{o}s--Falconer distance conjecture]\label{erdf}
 Let $\alpha(d)$ be the smallest exponent such that whenever $E \subseteq \mathbb{F}_q^d$ satisfies $|E| \gg q^{\alpha(d)}$, one has $|\Delta(E)| \gg q$. If $d \geq 2$ is even, then $\alpha(d) = \frac{d}{2}$.
\end{conjecture}

Here, and throughout the paper,  $A \ll B$ (or equivalently $A = O(B)$) means there exists an absolute constant $C > 0$ such that $A \leq CB$, $A \sim B$ means $c_1 B \leq A \leq c_2 B$ for some absolute constants $0 < c_1 \leq c_2$.

Iosevich and Rudnev \cite{IR07} initiated the study of this problem via Fourier analysis and proved $\alpha(d)\le\frac{d+1}{2}$. Hart, Iosevich, Koh, and Rudnev \cite{HIKR11} later showed that $\alpha(d)=\frac{d+1}{2}$ for odd $d\ge3$. In the plane $d=2$, the threshold was improved through a sequence of works \cite{BHIPR17,CEHIK12} to $\frac{4}{3}$, and more recently to $\frac{5}{4}$ over prime fields by Murphy, Petridis, Pham, Rudnev, and Stevens in \cite{MPPRS22}. However, the conjectured threshold of $\frac{d}{2}$ for all even dimensions $d \ge 2$ remains a challenging problem.

In a recent paper, Pham and Yoo \cite{phamYoo23} confirmed the rotational analog of the conjecture in the plane over prime fields by showing that for any $E, F\subset \mathbb{F}_q^2$ with $q$ an odd prime, if $|E|, |F|\ge q$, then for almost every rotation $\mathsf{g}$, one has $|\Delta(E, \mathsf{g}F)|\gg q$. Several applications of the distance problem on intersection patterns and expanding functions can also be found in \cite{phamYoo23}.

In \cite{IR07}, Iosevich and Rudnev also verified the $\frac{d}{2}$ threshold for \textit{Salem sets},  i.e.,  sets $E$ whose Fourier transform admits the optimal \(L^\infty\) bound at all non-zero frequencies: $\max_{\mathbf{m}\ne \mathbf{0}}|\widehat{E}(\mathbf{m})|\le q^{-d}|E|^{\frac{1}{2}}$. Fraser \cite{fraser2} recently introduced a flexible notion of $(u,s)$--Salem sets. In particular, for $u \in [1, \infty]$ and $s \in [0,1]$, we say $E$ is a \emph{$(u, s)$--Salem set} if 
	\begin{align}\label{Salem-set-inequality}
		\| \widehat{E} \|_u \ll q^{-d} |E|^{1-s},
	\end{align}
	where 
	\begin{equation}\label{def:norm-L^p}
		\| \widehat{E} \|_u:=
		\begin{cases}
			\bigg( \frac{1}{q^d} \sum\limits_{\mathbf{x} \in \mathbb{F}^d_q \backslash \{\mathbf{0}\}}  |\widehat{E}(\mathbf{x})|^u \bigg)^{\frac{1}{u}}, ~~~~~~~~~ u \in [1, \infty),\\
			\sup\limits_{\mathbf{x} \in \mathbb{F}^d_q \backslash \{\mathbf{0}\} } |\widehat{E}(\mathbf{x})|, ~~~~~~~~~~~~~~~~~~~~ u=\infty.
		\end{cases}
	\end{equation}

    This class is quite general: for every \(u\ge 2\), any set is automatically \((u,\frac{1}{u})\)--Salem. Fraser \cite{fraser2} initiated the systematic study of several problems in this setting, including sumset-type problems, the finite field distance problem, and the problem of counting $k$-simplices. More recently, in his joint work with Rakhmonov, this framework has been extended to exceptional projection problems and to Stein–Tomas type restriction theorems in finite fields \cite{fraser11, fraser22}. We refer to the recent survey \cite{fraser33} for a detailed account of these developments. In this paper, we investigate the distance set, its size and structure, within \((u,s)\)–Salem sets for even \(u\). The case of odd $u$ is of independent interest and will be treated in a sequel.



\subsection{Main question and conjectures}

For \(s\in[0,1]\) and even \(u=2k\) with \(k\in\mathbb N\), we begin with \((u,s)\)–Salem sets, which possess robust additive structures.  As can be seen from Corollary \ref{cor:euivalence-Salem-bound} in Section \ref{sec2},  the Salem condition
\[
\|\widehat{E}\|_{2k}\ \ll \ q^{-d}\,|E|^{1-s}
\]
is (up to absolute constants) equivalent to the \(2k\)-fold additive energy bound
\[
\Lambda_{2k}(E)\ \ll \ |E|^{2k(1-s)}\;+\;q^{-d}\,|E|^{2k}.
\]

We focus on \(u=4\); for even \(u\ge 6\), the same approach yields analogous statements. In particular,  \(E\subseteq\F_q^d\) is \((4,s)\)–Salem if and only if
\begin{equation}\label{eq:energy4}
\Lambda_{4}(E)\ \ll \ q^{-d}|E|^{4}\;+\;|E|^{4-4s}.
\end{equation}
The parameter restrictions for meaningful Salem sets emerge naturally from this observation. Since $\Lambda_4(E) \geq |E|^2$ trivially, the upper bound 
$|E|^{4-4s}$ is only nontrivial when $s \leq \frac{1}{2}$. Similarly, since $\Lambda_4(E) \leq |E|^3$, the bound $|E|^{4-4s}$ provides meaningful control only when $s \geq \frac{1}{4}$. Thus, throughout our analysis, we focus on the parameter range $s \in \left[\frac{1}{4}, \frac{1}{2}\right]$, which captures the most interesting and nontrivial behavior of $(4,s)$--Salem sets. Therefore, unless otherwise specified, we will always assume 
 $\frac{1}{4} \le s \le \frac{1}{2}$  throughout  this paper.

Fraser combined the Fourier analytic framework originally developed by
Iosevich and Rudnev in \cite{IR07} with the \((4,s)\)--Salem condition to
obtain the following distance estimate. The continuous counterpart of this result can be found in
\cite[Theorem 7.3]{JF1}.
\begin{theorem} [Fraser, \cite{fraser2}, Theorem 9.3] \label{Fraser}
If $ E\subseteq \mathbb F_q^d$  is $(4, s)$--Salem with $|E|\gg q^{\frac{d}{2}}$,  then 
\[ |\Delta(E)|\gg \min\{\,q,\, q^{1-d}|E|^{4s}\,\}.\]
This clearly implies that  if   $|E|\ge  q^{\frac{d}{4s}}$, then $|\Delta(E)|\gg q$. \end{theorem}

In this setting, the main question is:
\begin{problem} \label{mainQ}
Determine the smallest exponent \(\alpha=\alpha(d,4,s)\) such that for every \((4,s)\)–Salem set \(E\subseteq\F_q^d\) with \(|E|\gg q^{\alpha}\),  one has \(|\Delta(E)|\gg q\).
\end{problem}

Because of this generality, the problem is harder than the classical Erd\H{o}s–Falconer problem, and even formulating a conjecture is subtle: any threshold must be compatible with the known results for arbitrary sets and for those with controlled additive energy. Some representative examples appear in Propositions~\ref{prop11}--\ref{prop:mull} in Section \ref{sec2}.

\subsubsection*{Formulation of the conjecture} To formulate the conjecture, in Section \ref{constructions}, we construct several $(4, s)$--Salem sets with few distances:
\begin{enumerate}
\item If $d$ is even and $s\in\big[\tfrac14,\ \tfrac{d+2}{4d}\big)$, there exists a $(4,s)$--Salem set $E\subseteq\F_q^d$ with $|E|\sim q^{\frac{d}{2}}$ and $|\Delta(E)|=o(q)$.
\item If $d$ is even and $s\in\big[\tfrac{d+2}{4d},\ \tfrac12\big]$, there exists a $(4,s)$--Salem set $E\subseteq\F_q^d$ with $|E|\sim q^{\frac{d+2}{8s}}$ and $|\Delta(E)|=o(q)$.
\item If $d$ is odd and $s\in\big[\tfrac14,\ \tfrac12\big]$, there exists a $(4,s)$--Salem set $E\subseteq\F_q^d$ with $|E|\sim q^{\frac{d+1}{8s}}$ and $|\Delta(E)|=o(q)$.
\end{enumerate}

Hence, it is plausible to make the following conjecture.
\begin{conjecture}\label{conj}  Let $\alpha(d,4,s)$ denote the smallest number defined in Problem \ref{mainQ}. Then we have
\[
\alpha(d,4,s)=
\begin{cases}
1,& d=2,\ s\in\big[\tfrac14,\tfrac12\big],\\[3pt]
\frac{d}{2}, & d\ge4\ \text{even},\ s\in\big[\tfrac14,\tfrac{d+2}{4d}\big],\\[3pt]
\frac{d+2}{8s}, & d\ge4\ \text{even},\ s\in\big[\tfrac{d+2}{4d},\tfrac12\big],\\[3pt]
\frac{d+1}{8s}, & d\ge3\ \text{odd},\ s\in\big[\tfrac14,\tfrac12\big].
\end{cases}
\]
\end{conjecture}

\paragraph{Justification of Conjecture~\ref{conj}.}
For $j \in \mathbb F_q,$ we write
\[
S_j := \{x \in \mathbb F_q^d : \|x\| = j\}
\]
for the sphere of radius $j$.

We now explain why this conjecture is natural by showing that it agrees with all known results. 
\begin{itemize}
\item \textbf{Case $d = 2$:} The conjecture matches the Erd\H{o}s--Falconer distance conjecture. If $E$ lies on a circle of non-zero radius, then $\Lambda_4(E)\ll |E|^2$, so $E$ is $(4,s)$--Salem for all $s\le\tfrac12$.  By Lemma~\ref{lem-1} in Section \ref{constructions},  there are such sets $E$ with $|\Delta(E)|\sim |E|$, pointing to the threshold exponent $1$. 

\item \textbf{Case $d \geq 4$ even:} Hart, Iosevich, Koh, and Rudnev \cite{HIKR11} showed that for subsets $E\subseteq S_j$, $j\ne 0$, the condition $|E|\ge C q^{\frac{d}{2}}$ is both necessary and sufficient for $|\Delta(E)|\gg q$. Proposition~\ref{prop11}  in Section \ref{sec2} further shows that whenever
\[
q^{\frac{d-2}{4(1-2s)}}\ \ll\ |E|\ \ll\ q^{\frac{d}{2}},
\]
the set $E$ is $(4,s)$--Salem for all $s\le\tfrac{d+2}{4d}$. At $s=\tfrac{d+2}{4d}$, Conjecture \ref{conj} reproduces the $q^{\frac{d}{2}}$ threshold, which matches Hart, Iosevich, Koh, and Rudnev's theorem.

\item  \textbf{Case $d \geq 3$ odd:} Hart, Iosevich, Koh, and Rudnev \cite{HIKR11} also proved the sharp threshold $|E|>2 q^{\frac{d+1}{2}}$ for general $E\subset\F_q^d$. Since every set is $(4,\tfrac14)$--Salem and $\tfrac{d+1}{8s}=\tfrac{d+1}{2}$ at $s=\tfrac14$, the identity $\alpha(d,4,s)=\tfrac{d+1}{8s}$ is consistent with both known results and our constructions.
\end{itemize}

Next, we introduce the relationship between Conjecture \ref{conj} and the Erd\H{o}s--Falconer distance problem for sets contained in the sphere $S_j$.

It is conjectured in \cite{HIKR11} that if $E\subseteq S_j\subset \F_q^d$ with $d$ odd and $j\neq 0$, then the weaker size condition $|E|\gg q^{\frac{d-1}{2}}$ should already force $|\Delta(E)|\gg q$.

Proposition~\ref{propro1} in Section \ref{sec2} shows that if
\[
q^{\frac{d-1}{4(1-2s)}} \ \ll\ |E| \ \ll\ q^{\frac{d+1}{2}},
\]
then $E$ is $(4,s)$--Salem for all $s\le \frac{1}{4}+\frac{1}{2(d+1)}$. Taking $s=\frac{1}{4}+\frac{1}{4d}$, we have 
\[\frac{d-1}{4(1-2s)} = \frac{d+1}{8s}=\frac{d}{2}.\]
Thus, Conjecture~\ref{conj}$\!$ implies that the proposed $\frac{d-1}{2}$ threshold may fail in general.

If one further assumes that the sphere has primitive radius, Proposition~\ref{prop11} in Section \ref{sec2} implies that whenever
\[
q^{\frac{d-2}{4(1-2s)}}\ \ll\ |E|\ \ll\ q^{\frac{d}{2}},
\]
the set $E$ is $(4,s)$--Salem for all $s\le \tfrac{d+2}{4d}$. Taking $s=\tfrac{1}{4}+\tfrac{3}{8d-4}$, we have 
\[\frac{d-2}{4(1-2s)}=\frac{d+1}{8s}=\frac{d}{2}-\frac{1}{4}.\]
Hence, Conjecture \ref{conj} implies the exponent $\tfrac{d}{2}-\tfrac{1}{4}$ for this case.

If the energy estimate 
\begin{equation}\label{energy-ball}
\Lambda_4(E)\ \ll\ \frac{|E|^3}{q}\ +\ q^{\frac{d-3}{2}}|E|^2,
\end{equation}
holds, then  combining Proposition \ref{thm:L4-eps} in Section \ref{sec2}  and Conjecture \ref{conj}, we would get  the best threshold of this problem: $\frac{d-1}{2}$.

The preceding analysis leads to a conditional result derived from Conjecture~\ref{conj} for spheres with primitive radius. Specifically, we establish the following result.
\begin{theorem}\label{Kohsharp}
Suppose Conjecture \ref{conj} holds. Let $g$ be a primitive element of $\mathbb{F}_q$. Furthermore, assume that either $(d = 4k - 1$, $k\in \mathbb{N}$, and $q \equiv 1 \pmod{4})$ or $(d = 4k + 1$, $k\in \mathbb{N}).$ Let $E\subset S_g\subset \mathbb{F}_q^d$.
\begin{enumerate}
    \item[(i)] If $|E|\gg q^{\frac{d}{2}-\frac{1}{4}}$, then $|\Delta(E)|\gg q$.
    \item[(ii)] If the energy estimate \eqref{energy-ball} holds and $|E|\gg q^{\frac{d}{2}-\frac{1}{2}}$, then $|\Delta(E)|\gg q$.
\end{enumerate}
\end{theorem}

\subsection{Main results}
Our main contribution is a two-pronged improvement for $(4, s)$--Salem sets, with further refinements on spheres and applications to multiplicative subgroups and sets on an arbitrary variety. The first theorem reads as follows.

\begin{theorem}\label{thm1}
Let $d \geq 2$ and $s \in [\frac{1}{4}, \frac{1}{2}]$. Then
\[
\alpha(d,4,s)\ \le\ \min\Big\{\frac{d+2}{4s+1},\ \frac{d+4}{8s}\Big\}.
\]
In particular, 
\begin{itemize}
    \item[(i)] if $s\le \frac{1}{4}+\frac{1}{d}$, then $\alpha(d, 4, s)\le \frac{d+2}{4s+1}$.
    \item[(ii)] if $s\ge \frac{1}{4}+\frac{1}{d}$, then $\alpha(d, 4, s)\le \frac{d+4}{8s}$.
\end{itemize}
\end{theorem}

The proof of Theorem \ref{thm1} relies on two complementary
point--hyperplane incidence inequalities with multiplicities. The exponent
\(\frac{d+4}{8s}\) is obtained from an incidence bound whose hypotheses require
the point set to be \((4,s)\)--Salem and in which the associated hyperplane
multiset has \(L^2\)-norm controlled by \(\Lambda_4(E)\); in this part, the
Salem assumption is used twice. The exponent \(\frac{d+2}{4s+1}\) comes from a
companion incidence inequality that holds for arbitrary point sets, where the
only structural input is the \(L^2\) control of the hyperplane multiset through
\(\Lambda_4(E)\).

 Theorem \ref{thm1} improves Fraser's exponent of $\frac{d}{4s}$ for all $d \geq 4$ and $s \in [\frac{1}{4}, \frac{1}{2}]$. It also improves the Iosevich--Rudnev threshold of $q^{\frac{d+1}{2}}$ for $s \in [\frac{1}{4} + \frac{1}{2(d+1)}, \frac{1}{2}]$.

The proof of Theorem \ref{thm1} will be completed in Section \ref{secmain}. From that proof, we can easily see that, for a $(4, s)$-- Salem set $E$ in $\mathbb{F}_q^d$, the number of distinct distances satisfies
\[|\Delta(E)|\gg \min \left\lbrace q, ~~\max\left\lbrace \frac{|E|^{2s}}{q^{\frac{d}{4}}}, ~\frac{|E|^{\frac{1}{2}+2s}}{q^{\frac{d}{2}}}\right\rbrace \right\rbrace,\]
which improves the lower bound $\min\left\lbrace q, \frac{|E|^2}{q^d}\right\rbrace$ due to Iosevich and Rudnev in the above range of $s$. Moreover, with \textit{appropriate energy} estimates input, the exponent of $\frac{d+1}{2}$ can be recovered easily, see Subsection 4.3 for a discussion.

It follows from the inequality \eqref{eq:energy4} that any set $E\subset \mathbb F_q^d$ with $\Lambda_4(E)\ll |E|^2$  is a $(4, s)$--Salem set with $s=\frac{1}{2}$.  Hence, Theorem \ref{thm1} implies the following corollary. 
\begin{corollary}\label{multigroup}
    Let $E\subset \mathbb{F}_q^d$, $d\ge 4,$ with $\Lambda_4(E)\ll |E|^2$. If $|E|\ge C q^{\frac{d}{4}+1}$, then $|\Delta(E)|\gg q$.
\end{corollary}
The next corollary follows immediately from Theorem \ref{thm1} together with Proposition \ref{prop:mul} in Section \ref{sec2}.
\begin{corollary}\label{corK}
   Let $A$ be a multiplicative subgroup of $\mathbb{F}_q^*$  where $q$ is an odd prime. If $|A|\ge q^{\frac{9d+36}{28d}}$, then $|\Delta(A^d)|\gg q$.
\end{corollary}

 In \cite{kohkoh}, Koh, Pham, Shen, and Vinh proved that when $d = 4$, the condition $|A| \gg q^{\frac{1}{2}}(\log q)^{\frac{1}{6}}$ suffices to guarantee that $|\Delta(A^4)| \gg q$. For $d \geq 5$, their method yields the same threshold $|A| \gg q^{\frac{1}{2}}(\log q)^{\frac{1}{6}}$. By contrast, Corollary~\ref{corK} implies, for $d \geq 8$, the improved requirement $|A| \gg q^{\frac{9d+36}{28d}}$, since
\[
\frac{9d + 36}{28d} < \frac{1}{2}
\]
which holds for all integers $d \geq 8$. 
Thus, for fixed $d \geq 8$ and large $q$, we have $q^{\frac{9d+36}{28d}} \ll q^{\frac{1}{2}}(\log q)^{\frac{1}{6}}$, giving a strictly better asymptotic threshold.


Let $V\subseteq \mathbb F_q^{\,d}$ be a variety in $\mathbb{F}_q^d$. Its dimension is defined by
\[
  n:=\dim(V)=\max\{m:\ V \supseteq X_m \supseteq X_{m-1} \supseteq \cdots \supseteq X_0 \neq \varnothing\},
\]
where each $X_j$ is an irreducible subvariety of $V$.
If $V$ is irreducible with $\dim(V)=n$, its degree is
\[
  \deg(V)\ :=\ \sup\bigl\{\,|L\cap V|<\infty:\ L\subset\mathbb F_q^{\,d}
  \text{ is a }(d-n)\text{-dimensional affine subspace}\bigr\}.
\]
We also define the coset content of $V$
\[
  t(V)\ :=\ \max_{x\in \mathbb{F}_q^d,\ \Gamma\le \mathbb{F}_q^d}\ \bigl\{\,|\Gamma|:\ x+\Gamma\subseteq V\}.
\]

For all sufficiently large $q$ and every $E\subseteq V$ with $|E|=q^{\alpha}$, $t=q^{\ell}$, and $\alpha>\ell$, Proposition \ref{prop:mull} tells us that a set $E$, with $|E|=q^{\alpha}$, is a $(4, s)$--Salem set with 
\[\frac{1}{4}\le s\ \le\ \frac{1+\gamma(n)}{4}\ -\ \frac{\gamma(n)}{4}\cdot\frac{\ell}{\alpha},\]
where $\gamma(n) := \frac{1}{2^{n+1} - n - 2}$. Hence, the next corollary follows directly from Theorem \ref{thm1}.

\begin{corollary}
Let $V$ be a variety in $\mathbb{F}_q^d$ with
$n=\dim(V),$ $D=\deg(V)$, and $t=t(V)=q^\ell$.
For all sufficiently large $q$ and every $E\subseteq V$ with $|E|=q^\alpha$, \[\alpha\ge \min \Bigl\{\frac{d+2+\gamma(n)\,\ell}{\,2+\gamma(n)\,}\ ,\ \frac{d+4+2\,\gamma(n)\,\ell}{\,2+2\,\gamma(n)\,}\Bigr\}, \] we have $|\Delta(E)|\gg q$, where $\gamma(n) := \frac{1}{2^{n+1} - n - 2}$.
\end{corollary}

This result is nontrivial when 
\begin{equation}\label{eq:overlap-cond}
\gamma(n)\Bigl(1-\frac{\ell}{\alpha}\Bigr)\ \ge\ \frac{2}{d+1}.
\end{equation}

In general, establishing the optimal threshold $\alpha(d,4,s)$ from additive energy alone is difficult, and even the way one deploys energy estimates is subtle. Conjecture~\ref{conj} in even dimensions further indicates that any sharp result must treat the relevant ranges of $s$ separately.

Beyond the energy estimate, additional structural information about $E$ can lower the threshold. The following theorem improves the threshold when $E$ is contained in a sphere.

\begin{theorem}\label{thm13}

Let $d \geq 3$ and suppose that $E \subseteq S_j \subset \mathbb{F}_q^d$ with $j \neq 0$ is a $(4,s)$--Salem set. If 
$$|E| \gg q^{\frac{d+1}{4s+1}},$$
then $|\Delta(E)| \gg q$.
\end{theorem}
At $s=\tfrac{d+2}{4d}$, the exponent $\frac{d+1}{4s+1}=\frac{d}{2}=\frac{d+2}{8s}$, which matches Conjecture \ref{conj}. So, the theorem is sharp at this particular value of $s$. The proof of Theorem \ref{thm13} will be given in Section \ref{sec5}.


For $E, F\subseteq \mathbb F_q^d, d\ge 2,$  we denote 
\[\Delta(E, F)=\{\|\mathbf{x}-\mathbf{y}\|\colon \mathbf{x}\in E, ~\mathbf{y}\in F\}.\]  
One can note that  the proof of Theorem \ref{thm1} can be extended to the case of two different sets $E$ and $F$ with different structures or sizes. 
More precisely, we can obtain the following two theorems.  In  Section \ref{constructions} (Proposition \ref{sharptwoset}), we construct an example showing that these two results are generally sharp in even dimensions. 

\begin{theorem}\label{ex2}
 Let $E, F\subset \mathbb{F}_q^d$ with $d\ge 2$. If $E$ is a $(4, s_E)$--Salem set and $F$ is a $(4, s_F)$--Salem set, then  
   \[|\Delta(E, F)|\gg \min \left\lbrace q,~ \frac{|E|^{s_E}|F|^{s_F}}{q^{\frac{d}{4}}} \right\rbrace.\]
\end{theorem}
\begin{theorem}\label{ex1}
    Let $E, F\subset \mathbb{F}_q^d$ with $d\ge 2$. If $E$ is a $(4, s)$--Salem set and $F$ is an arbitrary set, then 
    \[|\Delta(E, F)|\gg \min \left\lbrace q,~|E|, ~\frac{|E|^{2s}|F|^{\frac{1}{2}}}{q^{\frac{d}{2}}} \right\rbrace.\]
\end{theorem}

\noindent\textbf{Organization of the paper.} The rest of this paper is organized as follows. In Section \ref{sec2}, we collect preliminary results on $(4,s)$--Salem sets and their relation to additive energy, and we present a collection of geometric examples, including subsets of paraboloids, spheres, multiplicative subgroups, and algebraic varieties, together with their corresponding energy estimates. Section \ref{constructions} provides explicit constructions of $(4,s)$--Salem sets with few distances, supporting Conjecture~\ref{conj} and demonstrating the sharpness of our main theorems in various parameter regimes. Section \ref{secmain} contains the proof of our main result, Theorem~\ref{thm1}, which we obtain via two complementary approaches that yield the exponents $\frac{d+4}{8s}$ and $\frac{d+2}{4s+1}$. In Section~\ref{sec5}, we prove Theorem~\ref{thm13}, which provides improved thresholds for sets contained in spheres by exploiting their additional geometric structure. Section~\ref{section-6} establishes a sharp point--hyperplane incidence bound for $(4,s)$--Salem sets. Finally, Section~\ref{section-7} discusses several open problems and directions for future research, including extensions to odd values of $u$, the solvability of systems of equations, packing set problems, and higher--dimensional generalizations.

\section{Preliminaries} \label{sec2}
In this section, we fix notation, collect some basic definitions, and present examples of $(4, s)$--Salem sets. We begin by introducing the key concepts related to $(4,s)$--Salem sets and their connection to additive energy, followed by comprehensive examples that illustrate the rich geometric structures underlying these sets.
\subsection{$(4, s)$--Salem sets and additive energy}

The notion of Salem sets provides a crucial bridge between Fourier analysis and combinatorial properties in finite field geometry. These sets are characterized by controlled Fourier decay, which translates directly into constraints on their additive energy. We start with the fundamental definition of the Fourier transform in our setting.

\begin{definition}
	Let $E\subset\mathbb{F}_q^d$. We define its Fourier transform by
	\begin{equation*}
		\widehat{E}(\mathbf{x}) \;=\; q^{-d}\sum_{\mathbf{y}\in E} \chi(-\mathbf{x} \cdot \mathbf{y}),
		\quad \mathbf{x}\in\mathbb{F}_q^d,
	\end{equation*}
	where $\chi(\cdot )$ is a nontrivial additive character on $\F_q$.
\end{definition}
Here, and throughout the paper,  we identify the set $E\subseteq \mathbb F_q^d$ with the indicator function $1_E$.


The connection between the $L^{2k}$-norm of the Fourier transform and additive energy is made precise through the following exact identity, which forms the cornerstone of our approach.
\begin{lemma}
For any positive integer $k \in \mathbb{N}$, we have 
\begin{equation}
\|\widehat{E}\|_{2k}^{2k} = q^{-d(2k)} \Lambda_{2k}(E) - q^{-d(2k+1)} |E|^{2k}.
\end{equation}
\end{lemma}

This identity reveals that the Salem condition on Fourier decay is essentially equivalent to controlling the additive energy of the set. The proof relies on orthogonality properties of additive characters and provides a direct computational tool for analyzing Salem sets.

\begin{proof}
We write $S(\mathbf{x})=\sum\limits_{\mathbf{y}\in E}\chi(-\mathbf{x}\cdot \mathbf{y})$ so that $\widehat{E}(\mathbf{x})=q^{-d}S(\mathbf{x})$.  By the orthogonality relations for additive characters, we obtain:
\begin{equation*}
		\sum_{\mathbf{x} \in\mathbb{F}_q^d}|S(\mathbf{x})|^{2k}
		\;=\;
		q^d\,\Lambda_{2k}(E).
	\end{equation*}
Since $S(\mathbf{0})=|E|$, separating the zero-frequency term yields:
\begin{equation*}
		\sum_{\mathbf{x} \neq \mathbf{0}}|S(\mathbf{x})|^{2k}
		\;=\;
		q^d\,\Lambda_{2k}(E)-|E|^{2k}.
	\end{equation*}

Substituting $S(\mathbf{x})=q^d\widehat{E}(\mathbf{x})$ and dividing by $q^d$ completes the proof.
\end{proof}

The immediate consequence of this identity is the equivalence between Salem conditions and additive energy bounds.
\begin{corollary}\label{cor:euivalence-Salem-bound}
	For any $s\in[0,1]$, the $(2k, s)$--Salem condition
	\begin{equation*}
		\|\widehat{E}\|_{2k}
		\;\ll\; q^{-d }\,|E|^{1-s}
	\end{equation*}
	is equivalent to the additive energy bound
	\begin{equation*}
		\Lambda_{2k}(E)
		\;\ll\;
		\; \frac{|E|^{2k}}{q^d}+ |E|^{2k(1-s)}.
	\end{equation*}
\end{corollary}

When specialized to the case $2k = 4$, this equivalence takes the particularly useful form
\begin{equation}\label{E4Salem}
\Lambda_4(E) \ll \frac{|E|^4}{q^d} + |E|^{4-4s}.
\end{equation}

This characterization allows us to work interchangeably between Fourier analytic conditions and combinatorial energy estimates, providing flexibility in our arguments.

\begin{remark} \label{GoodRemark}
 If  $|E|\ll q^{\frac{d}{4s}}$, then the condition  \eqref{E4Salem} is equivalent to $\Lambda_4(E)\ll |E|^{4-4s}.$
Hence,  $E\subseteq \mathbb F_q^d$ with $|E|\ll q^{\frac{d}{4s}}$ is a $(4, s)$--Salem set  if and only if \[\Lambda_4(E)\ll |E|^{4-4s}.\]
\end{remark}   


\subsection{Geometric examples and energy estimates}
This subsection provides a comprehensive catalog of $(4,s)$--Salem sets arising from natural geometric configurations in finite fields. These examples not only illustrate the theoretical framework but also demonstrate that our main theorems apply to concrete and well-studied objects in finite field geometry.

We recall several fundamental energy estimates for sets lying on classical varieties such as paraboloids and spheres.  Throughout this discussion, we use the notation

\begin{itemize}
\item $S_j := \{\mathbf{x} \in \mathbb{F}_q^d : \|\mathbf{x}\| = j\}$ for the sphere of radius $j$,
\item $P = \{\mathbf{x} \in \mathbb{F}_q^d : x_1^2 + \cdots + x_{d-1}^2 = x_d\}$ for the paraboloid.
\end{itemize}

The following lemmas collect  energy bounds from the literature, each capturing the geometric constraints imposed by the underlying variety.

\begin{lemma}[Lemma 2.2, \cite{IKL2020}]\label{lemma21}
Let $P \subset \mathbb{F}_q^d$ be the paraboloid with $d \geq 4$ being even. For $E \subseteq P$, we have
\[
\Lambda_4(E)\ll \frac{|E|^3}{q} + q^{\frac{d-2}{2}}|E|^2.
\]
\end{lemma}

\begin{lemma}[Theorem 3.1, \cite{KPV2021}]\label{lemma2.1}
Let $P \subset \mathbb{F}_q^d$ be the paraboloid with $d = 4k-1$, $k \in \mathbb{N}$ and $q \equiv 3 \pmod{4}$. For $E \subseteq P$,  we have
\[
\Lambda_4(E) \ll \frac{|E|^3}{q} + q^{\frac{d-2}{2}}|E|^2.
\]
\end{lemma}

\begin{lemma}[Lemma 3.1, \cite{IKLPS2021}]\label{lemma0}
Let $S_j$ be the sphere with non-zero radius $j \neq 0$ in $\mathbb{F}_q^d$ with $d \geq 4$ even. For $E \subseteq S_j$,  we have
\[
\Lambda_4(E) \ll \frac{|E|^3}{q} + q^{\frac{d-2}{2}}|E|^2.
\]
\end{lemma}

\begin{definition}
Let $\mathbb{F}_q$ be a finite field of order $q$ such that $q$ is an odd prime power. An element $g \in \mathbb{F}_q$ is called a \emph{primitive element} if $g$ is a generator of the group $\mathbb{F}_q^*$.
\end{definition}

\begin{lemma}[Theorem 3.5, \cite{KPV2021}]\label{lemma2.5}
Let $g$ be a primitive element in $\mathbb{F}_q$, and $S_g$ be the sphere of radius $g$ centered at the origin in $\mathbb{F}_q^d$ where
either $(d = 4k - 1, k\in \mathbb{N}, ~q \equiv 1 \pmod{4})$ or $(d = 4k + 1, k\in \mathbb{N})$. For $E \subseteq S_g$, we have
\begin{equation}\label{ene-e}
\Lambda_4(E) \ll \frac{|E|^3}{q} + q^{\frac{d-2}{2}}|E|^2.
\end{equation}
\end{lemma}

{\bf Sharpness:} Lemmas \ref{lemma21} and \ref{lemma0} are optimal. Extremal examples arise by taking
\(E\) to be a maximal affine subspace contained in the variety, see
\cite[Lemma~1.13]{KPV2021} and \cite[Lemma~2.1]{koh}. We have no evidence that
Lemma~\ref{lemma2.5} and Lemma~\ref{lemma2.1} are sharp. However, it follows from \cite[Lemma~1.13]{KPV2021} that it
cannot surpass
\begin{equation}\label{Energy-3}
\Lambda_4(E)\ \ll\ \frac{|E|^3}{q}\ +\ q^{\frac{d-3}{2}}|E|^2,
\end{equation}
which may be the correct bound in this setting.

In the following proposition, we consider the following geometric hypotheses.
Let $E\subset\mathbb{F}_q^d$. 
\begin{itemize}
  \item[(A)] $E\subseteq P$, where either $(d \geq 4 \mbox{ is even})$ or $(d = 4k-1, k\in \mathbb{N}, ~q \equiv 3 \pmod{4})$, 
  \item[(B)] $E\subseteq S_j$, where $j\ne0$ and $d\ge4$ is even, 
  \item[(C)]$E\subseteq S_g$, where $g$ is primitive and either $(d = 4k - 1, k\in \mathbb{N}, ~q \equiv 1 \pmod{4})$ or $(d = 4k + 1, k\in \mathbb{N})$.
\end{itemize}
In each of the cases (A), (B), and (C), we have 
\begin{equation}\label{SizePS} |P|\sim |S_j|\sim |S_g|\sim q^{d-1},\end{equation} 
and from Lemmas~\ref{lemma21}--\ref{lemma2.5}, the additive energy estimate 
\begin{equation}\label{E4K} \Lambda_4(E) \ll \frac{|E|^3}{q} + q^{\frac{d-2}{2}}|E|^2\end{equation} 
is satisfied.  
\begin{proposition}\label{prop11}
Let $E\subset\mathbb{F}_q^d, d\ge 3.$ Assume one of the geometric hypotheses (A), (B), and (C) holds.
Suppose moreover that $E$ satisfies one of the size conditions below
\begin{enumerate} 
\item[(i)] $ q^{\frac{d-2}{4(1-2s)}} \ll |E| \ll q^{\frac{d}{2}}, \qquad  \frac{1}{4} \leq s \le \frac{1}{4}+\frac{1}{2d}; 
$ \item[(ii)] $ q^{\frac{d}{2}} \ll |E| \ll q^{d-1}, \qquad \frac{1}{4} \leq s \le \frac{1}{4}+\frac{1}{4(d-1)}$; 
\item[(iii)] $q^{\frac{d}{2}} \ll |E| \ll q^{\frac{1}{4s-1}}, \qquad \frac{1}{4}+\frac{1}{4(d-1)} \leq s \le \frac{1}{4}+\frac{1}{2d}$; 
\item[(iv)] $ |E| \sim q^{d-1}$,\qquad $\frac{1}{4}\leq s\leq\frac{1}{2}$.
\end{enumerate}
\noindent Then $E$ is a $(4,s)$--Salem set.
\end{proposition}
To illustrate how these ranges determine the Salem parameter, we detail case (i); cases (ii)-(iv) follow analogously. Given $E$ with $|E|=q^\beta$, if there exists $s_\ast \le \frac{1}{4}+\frac{1}{2d}$ such that
\[
q^{\frac{d-2}{4(1-2s_\ast)}} \ll q^\beta \ll q^{\frac{d}{2}},
\]
then $E$ is $(4,s_\ast)$--Salem. Moreover, by monotonicity, if $E$ is $(4,s)$--Salem, then it is
$(4,s')$--Salem for all $s'\le s$. 

\begin{proof}  Recall from Corollary \ref{cor:euivalence-Salem-bound} that  $E$ is a $(4, s)$--Salem set  if and only if
\[ \Lambda_4(E) \ll \frac{|E|^4}{q^d}+ |E|^{4-4s} .\]  
Hence, to complete the proof,   it suffices from the estimate \eqref{E4K} to prove that  for  each of our assumptions, $(i), (ii), (iii), (iv),$ we have
\begin{equation} \label{LR} L:=  \frac{|E|^3}{q} + q^{\frac{d-2}{2}}|E|^2 \ll  \frac{|E|^4}{q^d}+ |E|^{4-4s}=: R. \end{equation}

\medskip
\noindent
\textbf{Case 1. Assume that $\frac{|E|^3}{q} \ll q^{\frac{d-2}{2}}|E|^2 (\iff |E| \ll q^{\frac{d}{2}})$}.
In this case,  since $L\sim q^{\frac{d-2}{2}}|E|^2, $ the statement \eqref{LR}  is equivalent to
\begin{equation}\label{Suf} q^{\frac{d-2}{2}}|E|^2  \ll  \frac{|E|^4}{q^d}+ |E|^{4-4s}.\end{equation}
We will show that  assumption $(i)$ is a sufficient condition for this inequality to hold.
Notice that such a sufficient condition is given by 
\[ q^{\frac{d-2}{2}}|E|^2 \ll \frac{|E|^4}{q^d}  \quad \mbox{or} \quad  q^{\frac{d-2}{2}}|E|^2 \ll  |E|^{4-4s},\]
which is  the same as the condition
$$  q^{\frac{3d-2}{4}} \ll |E|  \quad \mbox{or}\quad q^{\frac{d-2}{4(1-2s)} }\ll |E|.$$
Since $d\ge 3$, from our initial assumption $|E|\ll q^{\frac{d}{2}}$, we can ignore the condition $q^{\frac{3d-2}{4}} \ll |E|$ because this condition does not occur.  Therefore, as a sufficient condition for the inequality \eqref{Suf} to hold, we obtain the following condition
\[q^{\frac{d-2}{4(1-2s)} }\ll |E| \ll q^{\frac{d}{2}}.\]

For this inequality to be valid, the following condition must be satisfied: 
$\frac{d-2}{4(1-2s)} \le \frac{d}{2}$, which is the same as the condition $$\frac{1}{4} \leq s \le \frac{1}{4}+\frac{1}{2d},$$
where we also used the assumption that  $\frac{1}{4}\leq s \leq \frac{1}{2}.$
We conclude that  if assumption (i) holds, then  $E$ is a $(4,s)$--Salem set.  Hence, under assumption $(i),$ the proof of Proposition \ref{prop11} is complete.

\textbf{Case 2. Assume that $\frac{|E|^3}{q} \gg q^{\frac{d-2}{2}}|E|^2 (\iff |E| \gg q^{\frac{d}{2}})$}.

In this case,   $L\sim \frac{|E|^3}{q}$.   Hence,  the statement \eqref{LR}  is equivalent to the statement
\begin{equation}\label{ECT} \frac{|E|^3}{q}  \ll  \frac{|E|^4}{q^d}+ |E|^{4-4s}.\end{equation}

To complete the proof of Proposition \ref{prop11},  we will show that  each of assumptions $(ii), (iii),$ and $(iv)$ is a sufficient condition for the above inequality to hold.

If $|E|\sim q^{d-1}$ (assumption $(iv)$), then the inequality \eqref{ECT} clearly holds, because in this case $\frac{|E|^3}{q}\sim \frac{|E|^4}{q^d}$.
Hence,  the proof for the case of assumption $(iv)$ is completed.

To show that each of assumptions $(ii), (iii)$ is a sufficient condition for the inequality \eqref{ECT} to be satisfied, first observe that such a sufficient condition is given as follows $$\frac{|E|^3}{q} \ll |E|^{4-4s},$$
which is equivalent to 
$$ |E|\ll q^{\frac{1}{4s-1}}.$$
Combining this with the original assumption on $|E|$, $|E|\gg q^{\frac{d}{2}}$, the sufficient condition is as follows
$$  q^{\frac{d}{2}} \ll |E|\ll q^{\frac{1}{4s-1}}.$$

Since $|E|\ll q^{d-1}$ by \eqref{SizePS}, the sufficient condition becomes
$$  q^{\frac{d}{2}} \ll |E|\ll \min\{q^{\frac{1}{4s-1}},  ~~q^{d-1}\}.$$
Observing that 
\[ \min\left\{q^{\frac{1}{4s-1}},  ~~q^{d-1}\right\} = \begin{cases}
    q^{d-1}, \quad                  &\text{if } \frac{1}{4} \le s \le \frac{1}{4}+\frac{1}{4(d-1)},\\
     q^{\frac{1}{4s-1}}, \quad     & \text{if }  \frac{1}{4}+\frac{1}{4(d-1)} \le s \le \frac{1}{4}+ \frac{1}{2d}, 
    \end{cases}\]
we see that the condition $(ii)$ or $(iii)$ is a sufficient condition for the inequality \eqref{ECT} as required.
\end{proof}

The method in \cite{KPV2021} also implies the following optimal energy for all non-zero radii in $\mathbb{F}_q^d$ with $d$ odd.
\begin{lemma}\label{lemma2.6}
Suppose $d\ge 3$ is odd. For $E \subseteq S_j$ with $j\ne 0$, we have
\[
\Lambda_4(E) \ll \frac{|E|^3}{q} + q^{\frac{d-1}{2}}|E|^2.
\]
\end{lemma}

The proof of the proposition  below follows  the same structure  as that of Proposition \ref{prop11},   comparing the additive energy estimate in Lemma \ref{lemma2.6} with  the $(4, s)$--Salem set condition for the additive energy estimate. 
\begin{proposition}\label{propro1}
Assume that $E \subseteq S_j$, $j\ne 0$, with $d \ge 3$ odd, and that one of the following conditions holds
\begin{enumerate} 
    \item[(i)] $q^{\frac{d-1}{4(1-2s)}} \ll |E| \ll q^{\frac{d+1}{2}}, \qquad~~~\frac{1}{4}\leq s \leq \frac{1}{4}+\frac{1}{2(d+1)}$; 
    \item[(ii)] $q^{\frac{d+1}{2}} \ll |E| \ll q^{\,d-1}, \qquad ~~~~~~~ \dfrac{1}{4} \le s \le \dfrac{1}{4} + \dfrac{1}{4(d-1)}$;
    \item[(iii)] $q^{\frac{d+1}{2}} \ll |E| \ll q^{\frac{1}{4s-1}}$,\qquad~~~~~ $\frac{1}{4}+\frac{1}{4(d-1)} \leq s \leq \frac{1}{4}+\frac{1}{2(d+1)}$; 
    \item[(iv)] $|E|\sim q^{d-1} $,\qquad~~~~~~~~~~~~~~~~ $ \frac{1}{4} \leq s\le\frac{1}{2}$. 
\end{enumerate}
Then $E$ is a $(4,s)$\textnormal{-Salem set.}
\end{proposition}

\begin{proof}

By Corollary \ref{cor:euivalence-Salem-bound}, $E$ satisfies the $(4, s)$--Salem property precisely when
\[ \Lambda_4(E) \ll \frac{|E|^4}{q^d}+ |E|^{4-4s} .\]
Therefore, using Lemma  \eqref{lemma2.6}, our task reduces to verifying that under each of the assumptions $(i), (ii), (iii), (iv),$ the following holds
\begin{equation} \label{LR*} L:= \frac{|E|^3}{q} + q^{\frac{d-1}{2}}|E|^2 \ll \frac{|E|^4}{q^d}+ |E|^{4-4s}=: R. \end{equation}

\medskip
\noindent
\textbf{Case 1. Suppose $\frac{|E|^3}{q} \ll q^{\frac{d-1}{2}}|E|^2 (\iff |E| \ll q^{\frac{d+1}{2}})$}.
Under this assumption,  we see that $L\sim q^{\frac{d-1}{2}}|E|^2$. Hence,  \eqref{LR*} becomes
\begin{equation}\label{Suf*} q^{\frac{d-1}{2}}|E|^2 \ll \frac{|E|^4}{q^d}+ |E|^{4-4s}.\end{equation}
We establish that assumption $(i)$ guarantees this inequality.
Notice that  a sufficient condition for  \eqref{Suf*} is given by 
\[ q^{\frac{d-1}{2}}|E|^2 \ll \frac{|E|^4}{q^d}  \quad \mbox{or} \quad  q^{\frac{d-1}{2}}|E|^2 \ll  |E|^{4-4s},\]
which is  equivalent to the condition
$$  q^{\frac{3d-1}{4}} \ll |E|  \quad \mbox{or}\quad q^{\frac{d-1}{4(1-2s)} }\ll |E|.$$
Using  our initial assumption $|E|\ll q^{\frac{d+1}{2}}$ and disregarding the condition $q^{\frac{3d-1}{4}} \ll |E|, $  we take a sufficient condition for the inequality \eqref{Suf*} to hold as follows
\[q^{\frac{d-1}{4(1-2s)} }\ll |E| \ll q^{\frac{d+1}{2}}.\]

For this inequality to hold,  we  also need the condition 
$\frac{d-1}{4(1-2s)} \le \frac{d+1}{2}$, which becomes the condition 
$$\frac{1}{4} \leq s \le \frac{1}{4}+\frac{1}{2(d+1)}.$$
In conclusion,   assumption $(i)$ implies  the $(4,s)$--Salem set property.  Hence, under assumption $(i),$ the proof of Proposition \ref{propro1} is complete.

\textbf{Case 2. Assume that $\frac{|E|^3}{q} \gg q^{\frac{d-1}{2}}|E|^2 (\iff |E| \gg q^{\frac{d+1}{2}})$}.

In this case,   $L\sim \frac{|E|^3}{q}$.   Hence,  the statement \eqref{LR*}  is equivalent to the statement
\begin{equation}\label{ECT*} \frac{|E|^3}{q}  \ll  \frac{|E|^4}{q^d}+ |E|^{4-4s}.\end{equation}

To complete the proof of Proposition \ref{propro1},  it will be enough to show that   each of assumptions $(ii), (iii),$ and $(iv)$ is a sufficient condition for the above inequality to hold.

If $|E|\sim q^{d-1}$ (assumption $(iv))$, then the inequality \eqref{ECT*} obviously holds, because in this case $\frac{|E|^3}{q}\sim \frac{|E|^4}{q^d}$.
Hence,  the proof for the case of assumption $(iv)$ is completed.

To show that each of assumptions $(ii), (iii)$ is a sufficient condition for the inequality \eqref{ECT*} to hold, we need to check the condition $$\frac{|E|^3}{q} \ll |E|^{4-4s},$$
which is the same as 
$$ |E|\ll q^{\frac{1}{4s-1}}.$$
From the original assumption on $|E|$, $|E|\gg q^{\frac{d+1}{2}}$, the sufficient condition takes
$$  q^{\frac{d+1}{2}} \ll |E|\ll q^{\frac{1}{4s-1}}.$$

Since $|E|\ll q^{d-1}$ by \eqref{SizePS},    the sufficient condition  can be taken as
$$  q^{\frac{d+1}{2}} \ll |E|\ll \min\{q^{\frac{1}{4s-1}},  ~~q^{d-1}\}.$$
Observing that 
\[ \min\left\{q^{\frac{1}{4s-1}},  ~~q^{d-1}\right\} = \begin{cases}
    q^{d-1}, \quad                  &\text{if } \frac{1}{4} \le s \le \frac{1}{4}+\frac{1}{4(d-1)},\\
     q^{\frac{1}{4s-1}}, \quad     & \text{if }  \frac{1}{4}+\frac{1}{4(d-1)} \le s \le \frac{1}{4}+ \frac{1}{2(d+1)}, 
    \end{cases}\]
we see that the condition $(ii)$ or $(iii)$ is a sufficient condition for the inequality \eqref{ECT*} as required.
\end{proof}

If the energy bound \eqref{ene-e} is improved by a factor of $q^{-\frac{\epsilon}{2}}$, the same argument as in the previous proof yields the following ranges for $s$. We therefore omit the proof.

\begin{proposition}\label{thm:L4-eps}
Let $E \subset \mathbb{F}_q^d$. Assume $E\subseteq S_g$, where $g$ is primitive and either $(d = 4k - 1, k\in \mathbb{N}, ~q \equiv 1 \pmod{4})$ or $(d = 4k + 1, k\in \mathbb{N})$, and 
\[
\Lambda_4(E) \ll \frac{|E|^3}{q} + q^{\frac{d-2-\epsilon}{2}}|E|^2,
\]
for some $\epsilon > 0$ small. Then $E$ is a $(4,s)$--Salem set in the following ranges:
\begin{enumerate}
    \item[(i)] $
    q^{\frac{d-2-\epsilon}{4(1-2s)}} \ll |E| \ll q^{\frac{d-\epsilon}{2}},
    \qquad \frac{1}{4} \leq s \le \frac{1}{4}+\frac{1}{2(d-\epsilon)}
    $;
    \item[(ii)]   $
    q^{\frac{d-\epsilon}{2}} \ll |E| \ll q^{d-1},
    \qquad ~~  \frac{1}{4} \leq s \le \frac{1}{4}+\frac{1}{4(d-1)}
     $;
    \item[(iii)] $
    q^{\frac{d-\epsilon}{2}} \ll |E| \ll 
    q^{\frac{1}{4s-1}},
    \qquad ~~ \frac{1}{4}+\frac{1}{4(d-1)} \le s \le \frac{1}{4}+\frac{1}{2(d-\epsilon)};
     $
    \item[(iv)] $|E| \sim q^{d-1}, \qquad ~~~~~~~~~~~~~~ \frac{1}{4}\leq s \le\frac{1}{2} $.
\end{enumerate}
\end{proposition}

The next example of $(4, s)$--Salem sets is based on multiplicative subgroups of $\mathbb{F}_q^*$.
\begin{lemma}[Theorem 34, \cite{skredov}]\label{lemma:A-multiplicative}
    Let $A$ be a multiplicative subgroup of $\mathbb{F}_q^*$, where $q$ is an odd prime. If $|A|\le p^{\frac{3}{5}}$, then $\Lambda_4(A)\ll |A|^{\frac{22}{9}}\log |A|$.
\end{lemma}

\begin{proposition}\label{prop:mul}
     Let $A$ be a multiplicative subgroup of $\mathbb{F}_q^*$, where $q$ is an odd prime, with $|A|\le p^{\frac{3}{5}}$. Set $E=A^d$. Then $\Lambda_4(E)\ll |E|^{\frac{22}{9}}(\log |E|)^d$ and $E$ is $(4, s)$--Salem for all $s<\frac{7}{18}$.
\end{proposition}
\begin{proof}
    From Lemma \ref{lemma:A-multiplicative}, we have
    \begin{align*}
        \Lambda_4(E)=(\Lambda_4(A))^d \leq \bigg(|A|^{\frac{22}{9}}\log |A|  \bigg)^d \ll |E|^{\frac{22}{9}} \big(\log |E|  \big)^d. 
    \end{align*}
    Hence, $E$ is a $(4,s)$--Salem if 
    \begin{align*}
        |E|^{\frac{22}{9}} \big(\log |E|  \big)^d \ll \frac{|E|^4}{q^d}+|E|^{4-4s} \Leftrightarrow |E|^{\frac{22}{9}} \big(\log |E|  \big)^d \ll |E|^{4-4s} \Leftrightarrow \frac{22}{9}<4-4s \Leftrightarrow s<\frac{7}{18}.
    \end{align*}
This completes the proof.
\end{proof}

The next example of $(4, s)$--Salem sets is based on sets on arbitrary varieties in $\mathbb{F}_q^d$.
\begin{lemma}[Corollary 9, \cite{shkredov2}]\label{lemma-V-variety}
Let $V$ be a variety in $\mathbb{F}_q^d$ with
$n=\dim(V)$ and $D=\deg(V)$, and let $t=t(V)$.
For all sufficiently large $q$ and every $E\subseteq V,$ one has
\[
  \Lambda_4(E)\ \ll_{n,D}\ |E|^{3}\!
  \left(\frac{t}{|E|}\right)^{\!\frac{1}{2^{\,n+1}-n-2}}  \quad \mbox{for}~~n \geq 2,  \qquad \Lambda_4(E) \ll_{n,D} \frac{|E|^2}{t} \quad \mbox{for} ~~n = 1.
\]
\end{lemma}

\begin{proposition}\label{prop:mull}
    Let $V$ be a variety in $\mathbb{F}_q^d$ with
$n=\dim(V)\ge 2$ and $D=\deg(V)$, and let $t=t(V)=q^{\ell}$.
For all sufficiently large $q$ and every $E\subseteq V$ with $|E|>t$, one has $E$ is a $(4, s)$--Salem set with 
\[\frac{1}{4}\le s\ \le\ \frac{1+\gamma(n)}{4}\ -\ \frac{\gamma(n)}{4} \frac{\ell}{\alpha},\]
where $|E|=q^{\alpha}$ and $ \gamma(n):= \frac{1}{\,2^{\,n+1}-n-2\,}.$
\end{proposition}

\begin{proof}
    From Lemma \ref{lemma-V-variety}, we have that $E$ is a $(4,s)$--Salem if 
    \begin{align*}
        |E|^{3}\!
  \left(\frac{t}{|E|}\right)^{\!\frac{1}{2^{\,n+1}-n-2}} \ll \frac{|E|^4}{q^d}+|E|^{4-4s}.
    \end{align*}
    As $|E|=q^{\alpha}$, $|E|>t=q^{\ell}$ and $\gamma(n):= \frac{1}{\,2^{\,n+1}-n-2\,}$, the above inequality becomes 
    \begin{align*}
        q^{3\alpha-\gamma(n)(\ell-\alpha)} \ll q^{4\alpha-d} +q^{(4-4s)\alpha} .
    \end{align*}
    The dominant term on the right is $q^{(4-4s)\alpha}$ whenever  
\((4-4s)\alpha \ge 4\alpha-d\), which holds in the range of $s$ in the statement.
Thus, we require
\[
    3\alpha+\gamma(n)(\ell-\alpha)
    \le (4-4s)\alpha,
\]
and hence,
\[
    4s\alpha \le \alpha(1+\gamma(n)) - \gamma(n)\ell
    \quad\Longleftrightarrow\quad
    s \le \frac{1+\gamma(n)}{4}
      - \frac{\gamma(n)}{4}\,\frac{\ell}{\alpha}.
\]
Together with the trivial lower bound $s\ge \tfrac14$, this yields
\[
    \frac14 \le s \le 
    \frac{1+\gamma(n)}{4}
    - \frac{\gamma(n)}{4}\,\frac{\ell}{\alpha},
\]
which completes the proof.
\end{proof}

\section{Constructions of $(4, s)$--Salem sets with few distances}\label{constructions}
In this section, we construct $(4, s)$--Salem sets with $o(q)$ distances that support the sharpness of the exponents in Conjecture \ref{conj}. 
We  also construct Salem sets for the sharpness of Theorem \ref{ex2} and Theorem \ref{ex1}, which are results on the distance problem concerning two different sets. 
These constructions demonstrate the sharpness (or near-sharpness) of the proposed thresholds in various parameter regimes.
\begin{lemma}\label{lem-1}
	Let $p$ be an odd prime and $q=p^r$. Then there exists a set $E\subseteq\mathbb{F}_q^2$ with $|E|\sim p^{r-1}$ such that $\Lambda_4(E)\sim|E|^2$ and $|\Delta(E)|\sim p^{r-1}$.
\end{lemma}
\begin{proof}
	Suppose first that $q\equiv 3\pmod 4$. In this case, there are $q+1$ rotations in $\mathbb{F}_q^2$, which form a cyclic group of order $q+1$. Each such rotation can be written in the form 
\[
\begin{pmatrix} a & -b \\ b & a \end{pmatrix},
\]	
with $a^2+b^2=1$. 
	Let $\theta$ be a rotation of order $\frac{q+1}{p+1}\sim p^{r-1}$. Choose a point $x$ on the circle of radius $1$ centered at the origin, and define
	\[
    E=\{\theta^i x\colon 0\le i< \mathtt{ord}(\theta)\}.
    \]
Then, $|E|=\frac{q+1}{p+1}$, and one can verify that $|\Delta(E)|\sim \frac{q+1}{p+1}$ and $\Lambda_4(E)\sim |E|^2$.

In the case where  $q\equiv 1 \pmod{4}$,  the group of rotations has order $q-1$, and we can apply the same argument.
\end{proof}

\begin{lemma}\label{lem-2}
    If either $(d=4k$, $k \in \mathbb{N})$ or $(d=4k+2$, $k \in \mathbb{N}, ~q \equiv 1 \pmod {4}),$  then there exists a set $E\subseteq\mathbb{F}_q^d$ with $|E|=q^{\frac{d}{2}}$ such that $\Lambda_4(E)=|E|^3$ and $\Delta(E)=\{0\}$.
\end{lemma}
\begin{proof}
	By \cite[Lemma 5.1]{HIKR11}, there exist $\frac{d}{2}$ mutually orthogonal null vectors $\mathbf{v}_1,\dots, \mathbf{v}_{\frac{d}{2}}$ in $\mathbb{F}_q^d$. Define $E=\text{Span}_{\mathbb{F}_q}(\mathbf{v}_1,\dots, \mathbf{v}_{\frac{d}{2}})$. Then $E$ is a subspace of $\mathbb{F}_q^d$ with dimension $\frac{d}{2}$, so $|E|=q^{\frac{d}{2}}$. Since all vectors in $E$ are null and mutually orthogonal, we have $\Delta(E)=\{0\}$. Finally, because $E$ is a vector space, it follows that $\Lambda_4(E)=|E|^3$.
\end{proof}

We begin with the case where $d$ is even and $s \in \big[\tfrac{1}{4}, \tfrac{d+2}{4d}\big]$.
\begin{proposition}\label{cons1}
Let $q=p^r$ with $r$ sufficiently large. Assume either $(d = 4k$ for some $k \in \mathbb{N})$ or $(d = 4k + 2$ for some $k \in \mathbb{N}$ and $q \equiv 1 \pmod{4})$. Then, for any $s \in [\frac{1}{4}, \frac{d+2}{4d})$, there exists a $(4, s)$--Salem set $E \subseteq \mathbb{F}_q^d$ of size $\sim q^{\frac{d}{2}}$ such that $|\Delta(E)| = o(q)$ as $q \to \infty$.
\end{proposition}

\begin{proof}
By Lemma~\ref{lem-1}, there exists a set $A\subseteq\mathbb{F}_q^2$ with $|A|\sim p^{\,r-1}$, $\Lambda_4(A)=|A|^2$, and $|\Delta(A)|\sim p^{\,r-1}$. By  Lemma~\ref{lem-2}, there exists a set $B\subseteq\mathbb{F}_q^{\,d-2}$ with $|B|=q^{\frac{d-2}{2}}$, $\Lambda_4(B)=|B|^3$, and $\Delta(B)=\{0\}$. Now define
	\[
	E \;=\; A \times B 
	\;\subset\; \mathbb{F}_{q}^{d}
	\;=\; \mathbb{F}_{q}^{2} \times \mathbb{F}_{q}^{d-2}.
	\]
	We have $|A|\sim q^{1-\frac{1}{r}}$, and hence $|E|\sim q^{\frac{d}{2}-\frac{1}{r}}$. Moreover, $|\Delta(E)|=|\Delta(A)|\sim q^{1-\frac{1}{r}}= o(q)$ as $r \to \infty$. 
	For the additive energy, we compute
	\[
	\Lambda_4(E) 
	\;=\; \Lambda_4(A)\,\Lambda_4(B)
	\;=\; |A|^{2}\,|B|^{3}
	\;\sim\; q^{\,2(1-\frac{1}{r}) + \frac{3(d-2)}{2}}.
	\]
	On the other hand,
	\[
	\bigl(|A|\cdot|B|\bigr)^{\,4-4s}
	\;\sim\; q^{\,4(1-s)\!\left(1-\frac{1}{r} + \frac{d-2}{2}\right)}.
	\]
    One has $\Lambda_4(E)\le (|A| \cdot|B|)^{4-4s}$ when  
\[
2\left(1 - \frac{1}{r}\right) + \frac{3(d-2)}{2} \leq 4(1-s)\left(1 - \frac{1}{r} + \frac{d-2}{2}\right),
\]
which simplifies to
\[ s\le \frac{d+2}{4d}-\frac{d-2}{2d(dr-2)}\to \frac{d+2}{4d}.\]
As $r \to \infty$, this approaches $\frac{d+2}{4d}$. For $r$ sufficiently large and $s < \frac{d+2}{4d}$, by Corollary~\ref{cor:euivalence-Salem-bound}, $E$ is a $(4, s)$--Salem set, as required.
\end{proof}


Now we turn to the case where $d$ is even and $s \in \bigl[\tfrac{d+2}{4d}, \tfrac{1}{2}\bigr]$.
\begin{proposition}\label{cons2}
Let $q=p^r$ with $r$ sufficiently large. Assume either $(d = 4k$ for some $k \in \mathbb{N})$ or $(d = 4k + 2$ for some $k \in \mathbb{N}$ and $q \equiv 1 \pmod{4})$. Then, for any $s\in [\frac{d+2}{4d},\frac{1}{2}]$, there exists a $(4, s)$--Salem set $E\subseteq\mathbb{F}_q^d$ such that $|E|\sim q^{\frac{d+2}{8s}}$ and $|\Delta(E)|=o(q)$ when $q$ is large enough.
\end{proposition}

\begin{proof}
By Lemma~\ref{lem-1},  there exists a set $A \subseteq \mathbb{F}_q^2$ with $|A| \sim p^{r-1} =: q^\alpha$ (where $\alpha = 1 - \frac{1}{r}$) such that $\Lambda_4(A) = |A|^2$ and $|\Delta(A)| \sim p^{r-1} = o(q)$.

By Lemma~\ref{lem-2}, we can also construct a set $X \subseteq \mathbb{F}_q^{d-2}$ with $|X|= q^{\frac{d-2}{2}}$, $\Lambda_4(X)= |X|^3$, and $\Delta(X) = \{0\}$. From $X$, define a random subset $B$ by including each element of $X$ independently with probability $\theta=q^{\frac{2(1-2s)(\alpha+\frac{d-2}{2})+(1-\frac{d}{2})}{4s}}$, which is smaller than $1$ when $s\ge \frac{d+4\alpha-2}{4(d+2\alpha-2)}$.
Now define
	\[
	E \;=\; A \times B 
	\;\subseteq\; \mathbb{F}_{q}^{d}
	\;=\; \mathbb{F}_{q}^{2} \times \mathbb{F}_{q}^{d-2}.
	\]
Then $|E|\sim q^{\alpha + \frac{d-2}{2} + \frac{2(1-2s)(\alpha + \frac{d-2}{2}) + (1 - \frac{d}{2})}{4s}} =q^{\frac{d+4\alpha-2}{8s}}$. Since $\Delta(B) = \{0\}$, we obtain $|\Delta(E)| = |\Delta(A)| = o(q)$. 

	For the additive energy, we have
	\[
	\Lambda_4(E) 
	= \Lambda_4(A)\,\Lambda_4(B) 
	= |A|^2 \, |X|^3 \theta^4=q^{2\alpha+3 \cdot \frac{d-2}{2}+\frac{2(1-2s)\big(\alpha+\frac{d-2}{2}\big)+\big(1-\frac{d}{2}\big)}{s}}.
	\]
    On the other hand 
    \[ |E|^{4-4s}=q^{(4-4s) \cdot \frac{d+4\alpha-2}{8s}}.  \]
    For $\alpha<1$ and $s \in \big[\frac{1}{4}, \frac{1}{2} \big]$, we have
    \begin{align*}
        2\alpha+3 \cdot \frac{d-2}{2}+\frac{2(1-2s)\big(\alpha+\frac{d-2}{2}\big)+\big(1-\frac{d}{2}\big)}{s} \leq (4-4s) \bigg(\frac{d+4\alpha-2}{8s} \bigg).
    \end{align*}
    Together with the condition $s\ge \frac{d+4\alpha-2}{4(d+2\alpha-2)}$, 
    we obtain
    \[
        \Lambda_4(E)\le |E|^{\,4-4s}
        \qquad\text{whenever}\qquad
    s\ge \frac{d+4\alpha-2}{4(d+2\alpha-2)}.
    \]
    
Note that the condition $s\ge \frac{d+4\alpha-2}{4(d+2\alpha-2)}$  becomes $s\ge \frac{d+2}{4d}$ as $\alpha\to 1$,  (i.e., as $r \to \infty$). By Corollary~\ref{cor:euivalence-Salem-bound}, $E$ is a $(4,s)$--Salem set, as required.
	\end{proof}

Finally, we consider the case where $d$ is odd.
  \begin{proposition}\label{cons3}
	Let $d\ge3$ be an odd integer and $q = p^r$ with $r$ sufficiently large.  Assume either $(d=4k+1, k\in \mathbb{N})$ or $(d=4k+3, ~k\in \mathbb{N}, q\equiv 1\pmod 4)$. Then for any $s\in [\frac{1}{4},\frac{1}{2}]$, there exists a $(4, s)$--Salem set $E\subseteq\mathbb{F}_q^d$ such that 
    $|E|\sim q^{\frac{d+1}{8s}}$ and $|\Delta(E)| = o(q)$ as $q \to \infty$.
\end{proposition}
Note that by applying Proposition \ref{cons2} to the subspace $\mathbb{F}_q^{d-1}\times \{0\}$, we achieve the same conclusion with a smaller range of $s$, namely, $s\in [\frac{1}{4}, \frac{d+1}{4(d-1)}]$.
\begin{proof}
	Let $A\subseteq\mathbb{F}_q$ with $|A|\sim q^{\alpha}$, where $\alpha<1$ is chosen sufficiently close to $1$. In this case, we know that $\Lambda_4(A)\le|A|^3$ and $|\Delta(A)|\sim q^{\alpha}$. 
	
	By Lemma~\ref{lem-2}, we can construct a set $X \subseteq \mathbb{F}_q^{d-1}$ such that $|X| = q^{\frac{d-1}{2}}$, $\Lambda_4(X) = |X|^3$, and $\Delta(X) = \{0\}$. From $X$, define a random subset $B$ by retaining each element independently with probability $q^{\frac{(d+1)(1-4s)}{8s}}$. Since $s \ge \tfrac{1}{4}$, this exponent is nonpositive, so the choice of probability is valid.  
	
	By setting
	\[
	E \;=\; A \times B 
	\;\subseteq\; \mathbb{F}_{q}^{d}
	\;=\; \mathbb{F}_{q} \times \mathbb{F}_{q}^{d-1}.
	\]
    we have
	\[
	|E| \;\sim\; q^{\,\alpha + \frac{(d+1)(1-4s)}{8s} + \frac{d-1}{2}} 
	\;=\; q^{\frac{d+1}{8s}+\alpha-1}.
	\] 
	Moreover, since $\Delta(B) = \{0\}$, we obtain $|\Delta(E)| = |\Delta(A)|\sim q^\alpha$. 
	
	Next, we estimate the additive energy. Using the product structure of $E$, we have
	\[
	\Lambda_4(E) 
	= \Lambda_4(A)\,\Lambda_4(B) 
	\;\le\; |A|^3 \, |X|^3 \, q^{\frac{(d+1)(1-4s)}{2s}}
	\;\sim\; q^{\,3\alpha + \frac{d+1}{2s} - \frac{d}{2} - \frac{7}{2}}.
	\]
	On the other hand,
	\[
	(|A||B|)^{4-4s}
	\;\sim\; q^{(4-4s)\alpha+ \frac{(d+1)(1-s)}{2s} - 4 + 4s}.
	\]
	For $\alpha <1$ and $s \geq \frac{1}{4}$, we have 
    \[3\alpha + \frac{d+1}{2s} - \frac{d}{2} - \frac{7}{2}\leq (4-4s)\alpha+ \frac{(d+1)(1-s)}{2s} - 4 + 4s. \]
	This implies that $\Lambda_4(E) \le (|A||B|)^{4-4s}$ for all $s \in \big[\frac{1}{4}, \frac{1}{2}\big]$. Therefore, by Corollary~\ref{cor:euivalence-Salem-bound}, the set $E$ is a $(4,s)$--Salem set. Taking $\alpha\to 1$, the proof is complete. 
\end{proof}

We now introduce a proposition showing that Theorems \ref{ex2} and \ref{ex1}, which are results on the distance problem between two sets where at least one is a Salem set, cannot be improved in general even dimensions.
\begin{proposition}\label{sharptwoset}
      Assume either $(d=4k+2$, $k \in \mathbb{N})$ or $(d=4k$, $k \in \mathbb{N}, ~q \equiv 1$ $\pmod{4}$). 
    \begin{itemize}
        \item There exist a $\left(4, \frac{1}{4}+\frac{1}{2d}\right)$--Salem set $E\subset \mathbb{F}_q^d$ and a set $F\subset \mathbb{F}_q^d$ such that

        $\Delta(E, F)|=\frac{|E|^{2s}|F|^{\frac{1}{2}}}{q^{\frac{d}{2}}}.$
        \item There exist a $\left(4, s_E\right)$--Salem set $E\subset \mathbb{F}_q^d$ and a $\left(4, s_F\right)$--Salem set $F\subset \mathbb{F}_q^d$ with $s_E=\frac{1}{4}+\frac{1}{2d}$ and $s_F=\frac{1}{4}$
        such that
        $\Delta(E, F)|=\frac{|E|^{s_E}|F|^{s_F}}{q^{\frac{d}{4}}}.$
    \end{itemize}
\end{proposition}

\begin{proof}
   By Lemma \ref{lem-2}, there exist $\frac{d-2}{2}$ vectors $\mathbf{v}_1, \ldots, \mathbf{v}_{\frac{d-2}{2}}$ in $\mathbb{F}_q^{d-2}\times \{0\}\times \{0\}$ such that $\mathbf{v}_i\cdot \mathbf{v}_j=0$ for all $1\le i, j\le \frac{d-2}{2}$. Let $C_1$ be the circle centered at the origin of radius $1$ in $\mathbb{F}_q^2$. Set $E=\text{Span}_{\mathbb{F}_q}(\mathbf{v}_1,\dots, \mathbf{v}_{\frac{d-2}{2}})\times C_1\subset \mathbb{F}_q^d$ and $F=\text{Span}_{\mathbb{F}_q}(\mathbf{v}_1,\dots, \mathbf{v}_{\frac{d-2}{2}})\subset \mathbb{F}_q^{d-2}\times \{0\}\times \{0\}$. We have $\Lambda_4(E)\sim q^{\frac{3(d-2)}{2}}\cdot q^2$ and $\Lambda_4(F)=q^{\frac{3(d-2)}{2}}$. By a direct computation, we can check that $E$ is a $\left(4, \frac{1}{4}+\frac{1}{2d}\right)$--Salem set and $F$ is a $\left(4, \frac{1}{4}\right)$--Salem set. On the other hand, $\Delta(E, F)=\{1\}$, and 
   \[\frac{|E|^{2s}|F|^{\frac{1}{2}}}{q^{\frac{d}{2}}}=\frac{|E|^{s_E}|F|^{s_F}}{q^{\frac{d}{4}}}=1,\]
   which completes the proof.
\end{proof}

\section{Proof of main theorem (Theorem \ref{thm1})}\label{secmain}

We know from Fraser's theorem (Theorem \ref{Fraser}) that $|\Delta(E)| \sim q$ when $|E| \ge q^{\frac{d}{4s}}$. Hence, for the proof of Theorem \ref{thm1}, we may assume that a $(4, s)$--Salem set $E$ satisfies $|E| \le q^{\frac{d}{4s}}$. Therefore, the equivalent condition \eqref{eq:energy4} for a $(4, s)$--Salem set $E \subset \mathbb{F}_q^d$ can be written as
\begin{equation}\label{SalemAss}
\Lambda_{4}(E) \ll |E|^{4-4s},
\end{equation}
which we assume throughout the proof of Theorem \ref{thm1} below.

\subsection{Proof of Theorem \ref{thm1}--the exponent $\frac{d+4}{8s}$} \label{sec51}
For $t\in \mathbb{F}_q$, let $\nu(t)$ be the number of pairs $(\mathbf{x}, \mathbf{y})\in E\times E$ such that $\|\mathbf{x}-\mathbf{y}\|=t$. 

Notice that Theorem \ref{thm1} with the exponent $\frac{d+4}{8s}$ follows from the following proposition, the Cauchy-Schwarz inequality, and the assumption \eqref{SalemAss} that $\Lambda_4(E)\ll |E|^{4-4s}$.
\begin{proposition}\label{theorem-upper-bound0-nu^2}
We have the estimate
   \[\sum_{t\in \mathbb{F}_q}\nu(t)^2\le \frac{|E|^4}{q}+|E|^3+q^{\frac{d}{4}}|E|^{3-s}\Lambda_4(E)^{\frac{1}{4}}.\]
\end{proposition}
Thus, it suffices to prove Proposition \ref{theorem-upper-bound0-nu^2}. Before proving this, we establish the following key incidence inequality for $(4, s)$--Salem sets.

\begin{lemma}\label{lm-counting}
    Let $P \subset \Fqd$ be a $(4, s)$--Salem set and let $P'$ be a set of points $(\mathbf{a}, b)\in (\mathbb{F}_q^d\setminus\{\mathbf{0}\})\times \mathbb{F}_q$, and let $N(P, P')$ be the number of pairs $(\mathbf{x}, (\mathbf{a}, b))\in P\times P'$ such that $\mathbf{a}\cdot \mathbf{x}=b$. Then 
\[
N(P, P') \leq \frac{|P| \, |P'|}{q} + |P'|^{\frac{3}{4}} q^{\frac{d}{4}} |P|^{1-s}.
\]
\end{lemma}
\begin{proof}
For each pair $(\mathbf{a}, b)\in P'$, define $N_{(\mathbf{a}, b)}$ as the number of points $\mathbf{x}\in P$ such that $\mathbf{a}\cdot \mathbf{x}=b$.  
    We write 
\begin{align}
    N(P, P')&=\sum_{(\mathbf{a}, b)\in P'}N_{(\mathbf{a}, b)}=\sum_{(\mathbf{a}, b)\in P'} \bigg(N_{(\mathbf{a}, b)}-\frac{|P|}{q}+\frac{|P|}{q} \bigg)\nonumber\\
    &=\frac{|P| \,|P'|}{q}+\sum_{(\mathbf{a}, b)\in P'}\bigg(N_{(\mathbf{a}, b)}-\frac{|P|}{q} \bigg)\nonumber.
    \end{align}
By H\"{o}lder's inequality,  we have   
\begin{equation} \label{incid} N(P, P')\le \frac{|P| \,|P'|}{q}+|P'|^{\frac{3}{4}}\left(\sum_{(\mathbf{a}, b)\in P'} \bigg(N_{(\mathbf{a}, b)}-\frac{|P|}{q} \bigg)^4\right)^{\frac{1}{4}}.\end{equation}

In the rest of the proof, we focus on bounding the sum $\sum_{(\mathbf{a}, b)\in P'} \bigg(N_{(\mathbf{a}, b)}-\frac{|P|}{q} \bigg)^4$ from above. 

We first write
$$ N_{(\mathbf{a}, b)}= \sum_{\mathbf{x}\in P: \mathbf{a}\cdot \mathbf{x}=b} 1= \frac{1}{q} \sum_{t\in \mathbb F_q} \sum_{\mathbf{x}\in \mathbb F_q^d}  P(\mathbf{x}) \chi(-t(\mathbf{a}\cdot \mathbf{x}-b)).$$  
 Separating the sum over $t$ into the cases $t = 0$ and $t \neq 0$, we obtain that 
\[ N_{(\mathbf{a}, b)} - \frac{|P|}{q} = q^{d-1}\sum_{t \neq 0} \chi(tb) \widehat{P}(t\mathbf{a}). \]

Thus, 
\begin{align*}
    \sum_{(\mathbf{a}, b)\in P'}\left(N_{(\mathbf{a}, b)}-\frac{|P|}{q}\right)^4&=q^{4d-4}\sum_{(\mathbf{a}, b)\in P'}\left\vert \sum_{t\ne 0}\chi(tb)\widehat{P}(t\mathbf{a}) \right\vert^4\\
    &\le q^{4d-4}\sum_{\mathbf{a}\in \mathbb{F}_q^d\setminus\{\mathbf{0}\},~ b\in \mathbb{F}_q}\sum_{t_1, t_2, t_3, t_4\ne 0}\chi(b(t_1+t_2-t_3-t_4))\widehat{P}(t_1\mathbf{a})\widehat{P}(t_2\mathbf{a})\overline{\widehat{P}(t_3\mathbf{a})}\overline{\widehat{P}(t_4\mathbf{a})}\\
    &\le q^{4d-3}\sum_{\mathbf{a}\in \mathbb{F}_q^d\setminus\{\mathbf{0}\}}\sum_{t_1, t_2, t_3\ne 0, t_1+t_2-t_3\ne 0}\widehat{P}(t_1\mathbf{a})\widehat{P}(t_2\mathbf{a})\overline{\widehat{P}(t_3\mathbf{a})}\overline{\widehat{P}((t_1+t_2-t_3)\mathbf{a})}\\
    &\le q^{4d-3}\cdot q^3\cdot \left(\sum_{\mathbf{m} \neq 0} |\widehat{P}(\mathbf{m})|^4\right)\le q^d|P|^{4-4s},
\end{align*}
where we used the estimate that 
\[\sum_{\mathbf{m} \neq 0} |\widehat{P}(\mathbf{m})|^4\le q^{-3d}|P|^{4-4s}.\]


Plugging this bound into (\ref{incid}), we obtain the desired bound.    
\end{proof}
Note that if $P'$ is a multiset, the same argument implies the following generalization. 
\begin{lemma}\label{lm9}
    Let $P \subset \Fqd$ be a $(4, s)$--Salem set and let $P'$ be a multiset of points $(\mathbf{a}, b)\in (\mathbb{F}_q^d\setminus \{\mathbf{0}\})\times \mathbb{F}_q$, and let $N(P, P')$ be the number of pairs $(\mathbf{x}, (\mathbf{a}, b))\in P\times P'$ such that $\mathbf{a}\cdot \mathbf{x}=b$. Then 
\[
N(P, P') \leq \frac{|P| |P'|}{q} + \left(\sum_{(\mathbf{a}, b)\in \overline{P'}}m(\mathbf{a}, b)^{\frac{4}{3}}\right)^{\frac{3}{4}}q^{\frac{d}{4}} |P|^{1-s},
\]
where $\overline{P'}$ denotes the set of distinct elements in $P'$, $m(\mathbf{a}, b)$ is the multiplicity of $(\mathbf{a}, b)$ in $\overline{P'}$, and $|P'|=\sum_{(\mathbf{a}, b)}m(\mathbf{a}, b)$.
\end{lemma}
We are now ready to prove Proposition \ref{theorem-upper-bound0-nu^2}. 

\begin{proof}[Proof of Proposition \ref{theorem-upper-bound0-nu^2}]
    By the Cauchy-Schwarz inequality, we have 
    \[\sum_{t}\nu(t)^2\le |E|\cdot |\{(\mathbf{x}, \mathbf{y}, \mathbf{z})\in E^3\colon \|\mathbf{x}-\mathbf{z}\|=\|\mathbf{x}-\mathbf{y}\|\}|.\]
For each $t\in \mathbb{F}_q$, define
\begin{align*}
X_t &:= \{(\mathbf{y},\mathbf{z}) \in E \times E : \|\mathbf{y}\| - \|\mathbf{z}\| = t\},\\
N_t &:= |\{(\mathbf{x},\mathbf{y},\mathbf{z}) \in E \times X_t : \mathbf{x} \cdot (\mathbf{y}-\mathbf{z}) = t\}|,\\
U_t &:= \{\mathbf{u} = \mathbf{y} - \mathbf{z} : (\mathbf{y},\mathbf{z}) \in X_t\}, ~~m_t(u):= |\{(\mathbf{y},\mathbf{z}) \in X_t : \mathbf{y} - \mathbf{z} = \mathbf{u}\}|.
\end{align*}
Then
$$N_t = \sum_{\mathbf{u} \in U_t} m_t(\mathbf{u}) \cdot |\{\mathbf{x} \in E : \mathbf{x} \cdot \mathbf{u} = t\}|,$$
and 
\[\sum_{t\in \mathbb{F}_q}\nu(t)^2\le |E|\cdot \sum_{t\in \mathbb{F}_q}N_t.\]
We have following observations: 
$$\sum_t |X_t| = |E|^2,$$
and 
$$\sum_t \sum_{\mathbf{u} \in U_t} m_t(u)^2 \leq \Lambda_4(E),$$
with the equality when $E$ is a set on a sphere.

Let
\[P':=\{(\mathbf{u}, t)\colon \mathbf{u}\in U_t,  ~t\in \mathbb F_q\}\setminus \{(\mathbf{0}, 0)\}.\]
Note that the case of $\mathbf{u}=0$ and $t=0$ contributes at most $|E|^2$ to $\sum_{t}N_t$.

Applying Lemma \ref{lm9} for $P:=E$ and $P'$, we have $\sum_{t\in \mathbb{F}_q}N_t=N(P, P')+|E|^2$ and is bounded from above by
\[\sum_{t}N_t\le \frac{|E|^3}{q}+|E|^2+q^{\frac{d}{4}}|E|^{1-s}\left( \sum_t \sum_{\mathbf{u} \in U_t}m_t(\mathbf{u})^{\frac{4}{3}} \right)^{\frac{3}{4}}.\]
Note that 
\[\left( \sum_t \sum_{\mathbf{u} \in U_t} m_t(\mathbf{u})^{\frac{4}{3}} \right)^{\frac{3}{4}}\le \left( \sum_t \sum_{\mathbf{u} \in U_t} m_t(\mathbf{u}) \right)^{\frac{1}{2}}\cdot \left( \sum_t \sum_{\mathbf{u} \in U_t} m_t(\mathbf{u})^{2} \right)^{\frac{1}{4}}\le |E|\cdot \Lambda_4(E)^{\frac{1}{4}}.\]
This implies that 
\[\sum_{t}N_t\le \frac{|E|^3}{q}+|E|^2+q^{\frac{d}{4}}|E|^{2-s}\Lambda_4(E)^{\frac{1}{4}},\]
which completes the proof.
\end{proof}


This approach can also be extended to the distance set of two sets $E$ and $F$, where $E$ is a $(4, s_E)$--Salem set and $F$ is a $(4, s_F)$--Salem. For any $t\in \mathbb{F}_q$, let $\nu_{E, F}(t)$ be the number of pairs $(\mathbf{x}, \mathbf{y})\in E\times F$ such that $\|\mathbf{x}-\mathbf{y}\|=t$. Then, we have 
\begin{align*}
    \sum_{t\in \mathbb{F}_q}\nu_{E, F}(t)^2\le & \frac{|F|^2|E|^2}{q}+\min\{|E|^2|F|, ~|F|^2|E|\}+q^{\frac{d}{4}}|F|^{2-s_F}|E| \, \Lambda_4(E)^{\frac{1}{4}} \\
    \ll &\frac{|F|^2|E|^2}{q}+q^{\frac{d}{4}}|F|^{2-s_F}|E|^{2-s_E}.
\end{align*}
As a consequence, we obtain 
\[|\Delta(E, F)|\gg \min \left\lbrace q, \frac{|E|^{s_E}|F|^{s_F}}{q^{\frac{d}{4}}} \right\rbrace.\]
So, $|\Delta(E, F)|\gg q$ under $|E|^{s_E}|F|^{s_F}\gg q^{\frac{d+4}{4}}$. This completes the proof of Theorem \ref{ex2}.

\subsection{Proof of Theorem \ref{thm1}--the exponent $\frac{d+2}{4s+1}$} \label{sec6}
Theorem \ref{thm1} with the exponent $\frac{d+2}{4s+1}$ follows from the following proposition, the Cauchy-Schwarz inequality, and the assumption \eqref{SalemAss} that $\Lambda_4(E)\ll |E|^{4-4s}$.
\begin{proposition}\label{propo2}
We have the estimate
  \[\sum_{t\in \mathbb{F}_q}\nu(t)^2\le \frac{|E|^4}{q}+q^{\frac{d}{2}}|E|^{\frac{3}{2}}\Lambda_4(E)^{\frac{1}{2}}.\]
\end{proposition}
The proof of Theorem \ref{thm1} with the exponent $\frac{d+2}{4s+1}$ is identical to that of the exponent $\frac{d+4}{8s}$, except that we use the following lemma in place of Lemma \ref{lm9}. A proof can be found in \cite{VHKPV20}.
\begin{lemma}\label{lm99}
    Let $P \subset \Fqd$ be an arbitrary set and let $P'$ be a multiset of points $(\mathbf{a}, b)\in \mathbb{F}_q^d\times \mathbb{F}_q$, and let $N(P, P')$ be the number of pairs $(\mathbf{x}, (\mathbf{a}, b))\in P\times P'$ such that $\mathbf{a}\cdot \mathbf{x}=b$. Then 
\[
N(P, P') \leq \frac{|P||P'|}{q} + \left(\sum_{(\mathbf{a}, b)\in \overline{P'}}m(\mathbf{a}, b)^{2}\right)^{\frac{1}{2}}q^{\frac{d}{2}} |P|^{\frac{1}{2}},
\]
where $\overline{P'}$ denotes the set of distinct elements in $P'$, $m(\mathbf{a}, b)$ is the multiplicity of $(\mathbf{a}, b)$ in $\overline{P'}$, and $|P'|=\sum_{(\mathbf{a}, b)}m(\mathbf{a}, b)$.
\end{lemma}
This approach can be extended to the distance set of two sets $E$ and $F$, where $E$ is a $(4, s)$--Salem set and $F$ is an arbitrary set. For any $t\in \mathbb{F}_q$, let $\nu_{E, F}(t)$ be the number of pairs $(\mathbf{x}, \mathbf{y})\in E\times F$ such that $\|\mathbf{x}-\mathbf{y}\|=t$. Then, we have 
\[\sum_{t\in \mathbb{F}_q}\nu_{E, F}(t)^2\le \frac{|F|^2|E|^2}{q}+|F|^2|E|+q^{\frac{d}{2}}|F|^{\frac{3}{2}}\Lambda_4(E)^{\frac{1}{2}}.\]
As a consequence, we obtain 
\[|\Delta(E, F)|\gg \min \left\lbrace q,~|E|, ~ \frac{|E|^{2s}|F|^{\frac{1}{2}}}{q^{\frac{d}{2}}} \right\rbrace.\]
So, $|\Delta(E, F)|\gg q$ under $|E|^{2s}|F|^{\frac{1}{2}}\gg q^{\frac{d+2}{2}}$. This completes the proof of Theorem \ref{ex1}.

\subsection{Discussion}
Given $E\subset \mathbb{F}_q^d$, define 
\[E':=\{(\mathbf{x}, \|\mathbf{x}\|)\colon \mathbf{x}\in E\}\subset \mathbb{F}_q^{d+1}.\]
We denote the additive energy of $E'$ by $\Lambda_4(E')$.

If $E$ is fully contained in a sphere, then $\Lambda_4(E)=\Lambda_4(E')$. Otherwise, we would expect that $\Lambda_4(E')\le \Lambda_4(E)$. The proofs of Propositions \ref{theorem-upper-bound0-nu^2} and \ref{propo2} imply stronger upper bounds 
  \[\sum_{t\in \mathbb{F}_q}\nu(t)^2\le \frac{|E|^4}{q}+|E|^3+q^{\frac{d}{4}}|E|^{3-s}\Lambda_4(E')^{\frac{1}{4}},\]
  and
 \begin{equation}\label{eqnov24}\sum_{t\in \mathbb{F}_q}\nu(t)^2\le \frac{|E|^4}{q}+q^{\frac{d}{2}}|E|^{\frac{3}{2}}\Lambda_4(E')^{\frac{1}{2}}.\end{equation}
Since $E'$ is a subset on a paraboloid in $\mathbb{F}_q^{d+1}$, we know from Lemma \ref{lemma21} (with $d+1$ even) and Lemma \ref{lemma2.1} with ($d=4k-2$ and $q\equiv 3\pmod 4$) that 
\[\Lambda_4(E')\ll \frac{|E|^3}{q}+q^{\frac{d-1}{2}}|E|^2.\]
Plugging this bound into (\ref{eqnov24}), we recover the Iosevich-Rudnev exponent of $\frac{d+1}{2}$ on the distance set. 


\section{Proof of Theorem \ref{thm13}}\label{sec5}
To prove Theorem \ref{thm13}, we follow the same strategy as in the previous sections. Therefore, we sketch the main differences. When a set \(E\) lies \emph{on} a sphere, we can use extra structure: every line through the origin contains at most two (antipodal) points of \(E\), and \(\|\mathbf{x}\|\) is constant for all \(\mathbf{x}\in E\). These properties yield a better exponent. As in the previous sections, we consider
\[
\|\mathbf{x}-\mathbf{y}\|=\|\mathbf{x}-\mathbf{z}\| \quad\Longrightarrow\quad \mathbf{x}\cdot(\mathbf{y}-\mathbf{z})=0,
\]
using that \(\|\mathbf{y}\|=\|\mathbf{z}\|\) on the sphere. We then apply Lemma \ref{lm99} with
\[
P=\{\lambda\,\mathbf{a}:\ \mathbf{a}\in E,\ \lambda\in\mathbb{F}^{*}\},
\qquad
P'=\{(\mathbf{a}-\mathbf{b},\,0):\ \mathbf{a},\mathbf{b}\in E\}.
\]
If we denote the multiplicity of $(\mathbf{u}, t)\in \overline{P'}$ by $m(\mathbf{u}, t)$, then
\[
\sum_{(\mathbf{u},t)\in \overline{P'}} m(\mathbf{u},t)^{2}
\;=\;
\Lambda_4(E).
\]
From here, Lemma \ref{lm99} completes the proof.


\section{A sharp point-hyperplane incidence bound for Salem sets}\label{section-6}

In this section, we establish a sharp point–hyperplane incidence bound for point sets that are $(4,s)$–Salem, which is a direct consequence of Lemma \ref{lm-counting}. The result is of independent interest with potential applications in incidence geometry.

\begin{theorem}\label{nonzero}
Let $P \subset \Fqd$ be a $(4, s)$--Salem set and let $H$ be a set of
hyperplanes in $\Fqd$ defined by $\mathbf{a}\cdot \mathbf{x}=b$ with $\mathbf{a}\ne \mathbf{0}$ and $b\ne 0$. Then
\[
I(P, H) \leq \frac{|P| \, |H|}{q} + |H|^{\frac{3}{4}} q^{\frac{d-1}{4}} |P|^{1-s}.
\]
\end{theorem}

\begin{remark}
    If $H$ contains hyperplanes with $b=0$, then the proof of Theorem \ref{nonzero} implies a weaker result, namely, 
    \[
I(P, H) \leq \frac{|P| |H|}{q} + |H|^{\frac{3}{4}} q^{\frac{d}{4}} |P|^{1-s}.
\]
\end{remark}

\begin{proof} 
The proof makes use of Lemma \ref{lm-counting} with the following crucial geometric observation: 
\[\mathbf{a}\cdot\mathbf{x}=b \quad\iff\quad (\lambda \mathbf{a})\cdot\mathbf{x}=(\lambda b),\]
for all $\lambda\ne 0$. We identify each hyperplane defined by $\mathbf{a} \cdot \mathbf{x} = b$ with $(\mathbf{a}, b)\in \mathbb{F}_q^{d+1}$. Set $P'=\{(\lambda \mathbf{a}, \lambda b)\colon (\mathbf{a}, b)\in H, \lambda\in \mathbb F_q^*\}$. A direct computation shows that 
\[I(P, H)=\frac{1}{q-1}N(P, P').\]
Applying Lemma \ref{lm-counting}, we obtain 
  \[
I(P, H) \leq \frac{|P||H|}{q} + |H|^{\frac{3}{4}} q^{\frac{d-1}{4}} |P|^{1-s}.
\]
This completes the proof.
\end{proof}

This incidence theorem is sharp in the sense that the term $|H|^{\frac{3}{4}} q^{\frac{d-1}{4}} |P|^{1-s}$ cannot be improved to $|H|^{\frac{3}{4}} q^{\frac{d-1-\epsilon}{4}} |P|^{1-s}$ for any $\epsilon>0$. To see this, let $H_0$ be a hyperplane in $\mathbb{F}_q^d$, and $H = \{H_0\}$, $P=H_0$. So $|H| = 1$, and 
\[
I(P, H) = |P| = q^{d-1}.
\]
Note that $P$ is $(4, \frac{1}{4})$--Salem. Theorem \ref{nonzero} tells us that
\[
I(P, H) \leq \frac{|P||H|}{q} + |H|^{\frac{3}{4}} q^{\frac{d-1}{4}} |P|^{1-s},
\]
which is at most 
\[\frac{|P|}{q}+q^{\frac{d-1}{4}}|P|^{\frac{3}{4}}=q^{d-2}+q^{d-1}.\]
This matches the lower bound of $q^{d-1}$ up to a negligible term, confirming sharpness.

\section{Open problems}\label{section-7}
Given the generality of the \((4,s)\)–Salem framework, it is natural to study other topics in this setting. We propose the following directions for future research.

\subsection{Odd values of $u$}

In the present paper, we focused on even values of $u$ (specifically $u = 4$) because the Salem condition $\|\widehat{E}\|_u \ll q^{-d}|E|^{1-s}$ translates cleanly into additive energy bounds. For odd values of $u$, this connection is less direct, and new techniques may be required.

\begin{problem}
Study the distance problem for $(3, s)$--Salem sets.
\end{problem}

\subsection{Solvability of systems of equations}

The solvability of systems of polynomial equations over finite fields is a classical problem with connections to number theory and algebraic geometry. Some references include~\cite{revista, sol4, hh2, sol5, sol3, sh2, sol1, sol2}.

\begin{problem}
Investigate the solvability of systems of bilinear and quadratic equations within $(4, s)$--Salem sets. Can the additional structures provided by the Salem condition yield improved bounds on the number of solutions?
\end{problem}

\subsection{Packing set problems}
Let $T$ be a set of transformations from $\mathbb{F}_q^d$ to $\mathbb{F}_q^d$. Define $T(E)=\{f(E)\colon f\in T\}$. 
The packing set problem asks for bounds on the size of $T(E)$ in terms of the size of $E$. Some references include~\cite{pack2, pack1}.

\begin{problem}
Study the packing set problem when $E$ is a $(4, s)$--Salem set.
\end{problem}

\subsection{Sharpness of energy estimates}

In the present paper, we relied on energy estimates from~\cite{IKL2020, IKLPS2021,KPV2021}. For spheres of primitive radius in odd dimensions, the sharpness of these bounds remains open.

\begin{problem}
Determine whether the energy bound~\eqref{Energy-3} holds for primitive-radius spheres in odd dimensions. If true, this would yield the $\frac{d-1}{2}$ threshold mentioned in Theorem~\ref{Kohsharp}.
\end{problem}

\subsection{Higher-dimensional generalizations}

Our results focused on distance sets defined by the quadratic form $\|\mathbf{x}\|= x_1^2 + \cdots + x_d^2$. It would be interesting to extend these results to other norms or distance functions. A reference can be found in \cite{AK, ks}.

\begin{problem} For $E\subset \mathbb{F}_q^d$ and an integer $k\ge 3$, define 
\[
\Delta_k(E) := \{\|\mathbf{x} - \mathbf{y}\|_k : \mathbf{x}, \mathbf{y} \in E\}, ~~\|\mathbf{x}\|_k=x_1^k+\cdots+x_d^k.
\]

Determine the smallest exponent \(\alpha=\alpha(d,4,s)\) such that for every \((4,s)\)–Salem set \(E\subseteq\F_q^d\) with \(|E|\gg q^{\alpha}\),  one has \(|\Delta_k(E)|\gg q\).
\end{problem}

{\bf Acknowledgement.} D. Cheong was supported by the research year program
of Chungbuk National University in 2025
D. Koh was supported by the National Research Foundation of Korea (NRF) grant funded by the Korea
government (MSIT) (NO. RS-2023-00249597). D. T. Tran was supported by the research project QG.25.02 of Vietnam National University, Hanoi. 
He also would like to thank the Vietnam Institute for Advanced Study in Mathematics (VIASM) for the hospitality and for the excellent working condition.

Chungbuk National University, Department of Mathematics, Chungdae-ro 1, Seowon-Gu, Cheongju City, Chungbuk 28644, Korea

Email: daewoongc@chungbuk.ac.kr

\vspace{0.5 cm}

School of Mathematical Sciences, Capital Normal University, Beijing 100048, China. 

Email: gnge@zju.edu.cn

\vspace{0.5 cm}

Chungbuk National University, Department of Mathematics, Chungdae-ro 1, Seowon-Gu, Cheongju City, Chungbuk 28644, Korea

Email: koh131@chungbuk.ac.kr

\vspace{0.5 cm}

Institute of Mathematics and Interdisciplinary Sciences at Xidian University, China. 

Email: thangpham.math@gmail.com

\vspace{0.5 cm}

VNU University of Science, Hanoi, Vietnam.

Email: tranthedung56@gmail.com

\vspace{0.5 cm}

Institute of Mathematics and Interdisciplinary Sciences at Xidian University, China. 

Email: zhangtao03@xidian.edu.cn


\begin{thebibliography}{00}
\bibitem{VHKPV20}
D. N. V. Anh, L. Ham, D. Koh, T. Pham, and L. A. Vinh, \textit{On a theorem of Hegyvári and Hennecart}, Pacific Journal of Mathematics, \textbf{305} (2) (2020), 407--421.


\bibitem{BHIPR17}
M. Bennett, D. Hart, A. Iosevich, J. Pakianathan, and M. Rudnev,
\textit{Group actions and geometric combinatorics in $\Bbb{F}_q^d$}, Forum Mathematicum, \textbf{29}(1) (2017), 91--110.






\bibitem{CEHIK12}
J. Chapman, M. Burak {Erdo\u{g}an}, D. Hart, A. Iosevich, and D. Koh,
\textit{Pinned distance sets, {$k$}-simplices, {W}olff's exponent in finite fields and sum-product estimates}, Mathematische Zeitschrift, \textbf{271}(1-2) (2012), 63--93.






\bibitem{revista}
D. Cheong, D. Koh, T. Pham, and L. A. Vinh, \textit{Distribution of the determinants of sums of matrices}, Revista Matematica Iberoamericana \textbf{37}(4) (2020),  1365--1398.



 \bibitem{Du3}
 X. Du, A.  Iosevich, Y. Ou, H. Wang, and R.  Zhang,  \textit{An improved result for Falconer's distance set problem in even dimensions}, Mathematische Annalen, \textbf{380} (2021), 1--17.

\bibitem{Du4}
X. Du, Y. Ou, K. Ren, and R.  Zhang, \textit{Weighted refined decoupling estimates and application to Falconer distance set problem}, 
\href{https://arxiv.org/abs/2309.04501}{arXiv:2309.04501 [math.CA]}, (2023).





\bibitem{falconer}
K. J. Falconer, \textit{On the Hausdorff dimensions of distance sets}, Mathematika, \textbf{32}(2) (1985), 206--212.

\bibitem{fraser2}
J. M. Fraser, \textit{$L^p$ averages of the Fourier transform in finite fields}, 
\href{https://arxiv.org/abs/2407.08589}{arXiv:2407.08589 [math.CO]}, accepted in Indiana University Mathematics Journal, 2026.


\bibitem{JF1}
J. M. Fraser, \textit{The Fourier spectrum and sumset type problems}, Mathematische Annalen, \textbf{390} (2024), 3891--3930.

\bibitem{fraser11}
J. M. Fraser and F. Rakhmonov, \textit{Exceptional projections in finite fields: Fourier analytic bounds and incidence geometry},  \href{https://arxiv.org/abs/2503.15072}{arXiv:2503.15072 [math.CO]}, (2025).

\bibitem{fraser22}
J. M. Fraser and F. Rakhmonov, \textit{An improved $L^2$ restriction theorem in finite fields}, appear in Proceedings of the American Mathematical Society,
\href{https://arxiv.org/abs/2505.09293}{arXiv:2505.09293 [math.CO]}, (2025).



\bibitem{fraser33}
J. M. Fraser and F. Rakhmonov, \textit{$L^p$ averages of the discrete Fourier transform and applications}, 
\href{https://arxiv.org/abs/2510.13483}{arXiv:2510.13483 [math.CO]}, (2025).





\bibitem{alex-fal}
L. Guth, A. Iosevich, Y. Ou, and H. Wang, \textit{On Falconer's distance set problem in the plane}, Inventiones mathematicae, \textbf{219}(3) (2019), 779--830.

\bibitem{sol4}
K. Gyarmati and A. S\'{a}rk\"{o}zy, \textit{Equations in finite fields with restricted solution sets,} II (algebraic equations), Acta Mathematica Hungarica, \textbf{119} (2008), 259--280.

\bibitem{HLR16}
B. Hanson, B. Lund, and O. Roche-Newton,
\textit{On distinct perpendicular bisectors and pinned distances in finite fields}, Finite Fields and their Applications, \textbf{37} (2016), 240---264.

\bibitem{HIKR11} 
 D. Hart, A. Iosevich, D. Koh, and M. Rudnev, \textit{Averages over hyperplanes, sum--product theory in vector spaces over finite fields and the Erd\H{o}s--Falconer distance conjecture}, Transactions of the American Mathematical Society, \textbf{363}(6) (2011), 3255--3275.


\bibitem{hh2}
N. Hegyv\'ari and F. Hennecart, \textit{A structure result for bricks in Heisenberg groups}, Journal of Number Theory, \textbf{133}(9) (2013), 2999--3006.

\bibitem{hh}
N. Hegyv\'ari and F. Hennecart, \textit{Expansion for cubes in the Heisenberg group}, Forum Mathematicum, \textbf{30}(1) (2018), 227--236. 



\bibitem{pack2}
N. Hegyv\'ari, L. Q. Hung, A. Iosevich, and T. Pham,  \textit{Packing sets under finite groups via algebraic incidence structures}, 
\href{https://arxiv.org/abs/2411.05377}
{arXiv:2411.05377 [math.CO]} (2024).


\bibitem{pack1}
L. Q. Hung, T. Pham, and K. Slavov, \textit{Sets preserved by a large subgroup of the special linear group},  Bulletin of the London Mathematical Society, \textbf{58}(2) (2026), e70298.








\bibitem{IR07}
A. Iosevich and M. Rudnev, {\em Erd\H{o}s--Falconer   distance problem in vector spaces over finite fields}, Transactions of the American Mathematical Society, \textbf{359}(12) (2007), 6127--6142.
\bibitem{AK}
A. Iosevich and D. Koh, \textit{The Erd\H{o}s–Falconer distance problem, exponential sums, and Fourier analytic approach to incidence theorems in vector spaces over finite fields}, SIAM Journal on Discrete Mathematics, \textbf{23}(1) (2009), 123--135.


\bibitem{IKL2020}
A. Iosevich, D. Koh, and  M. Lewko, \textit{Finite field restriction estimates for the paraboloid in high even dimensions}, Journal of Functional Analysis, \textbf{278}(11) (2020), 108450.

\bibitem{IKLPS2021}
A. Iosevich, D. Koh, S. Lee, T. Pham, and C.-Y. Shen, \textit{On restriction estimates for the zero radius sphere over finite fields}, Canadian Journal of Mathematics, \textbf{73}(3) (2021), 769--786.
\bibitem{entropy}
A. Iosevich, T. Pham, N. D. Quan, S. Senger, and B.  Xue, \textit{On an entropy inequality for quadratic forms and applications}, 
\href{https://arxiv.org/abs/2507.15196}
{arXiv:2507.15196 [math.CA]} (2025).

\bibitem{koh}
D. Koh, \textit{Conjecture and improved extension theorems for paraboloids in the finite field setting}, Mathematische Zeitschrift, \textbf{294} (2020), 51--69.

\bibitem{ks}
D. Koh and C-Y. Shen, \textit{The generalized Erd\H{o}s–Falconer distance problems in vector spaces over finite fields}, Journal of Number Theory, \textbf{132}(11) (2012), 2455--2473.

\bibitem{kohkoh}
D. Koh, T. Pham, C. -Y. Shen, and L. A. Vinh, \textit{A sharp exponent on sum of distance sets over finite fields}, Mathematische Zeitschrift, \textbf{297}(3)  (2021), 1749--1765.

\bibitem{KPV2021}
D. Koh, T. Pham, and L. A. Vinh, \textit{Extension theorems and a connection to the Erd\H{o}s--Falconer distance problem over finite fields}, Journal of Functional Analysis, \textbf{281}(8) (2021), 109137.







\bibitem{MPPRS22}
B. Murphy, G. Petridis, T. Pham, M. Rudnev, and S. Stevens, \textit{On the pinned distances problem over finite fields}, Journal of the London Mathematical Society, \textbf{105}(1) (2022), 469--499.


\bibitem{phamYoo23}
T. Pham and S. Yoo, \textit{Intersection patterns and connections to distance problems}, 
\href{https://arxiv.org/abs/2304.08004}
{arXiv:2304.08004 [math.CO]} (2023).

\bibitem{sol5}
T. Pham, S. Senger, N. T. Tuan, N. D. Thang, and L. A. Vinh, \textit{On the Solvability of Systems of Equations Revisited}, Vietnam Journal of Mathematics (2025), 1--13.



\bibitem{skredov}
 I. D. Shkredov, \textit{Some new inequalities in additive combinatorics}, Moscow Journal of Combinatorics and Number Theory, \textbf{3}(3) (2013), 189--239.

\bibitem{shkredov2}
 I. D. Shkredov, \textit{On multiplicative energy of subsets of varieties}, Canadian Journal of Mathematics,  \textbf{75}(1) (2023), 322--340.

\bibitem{sol3}
I. E. Shparlinski, \textit{On the solvability of Bilinear Equations in Finite Fields,} Glasgow Mathematical Journal, \textbf{50} (2008), 523–529.

\bibitem{sh2}
I. E. Shparlinski, \textit{On the additive energy of the distance set in finite fields}, Finite
Fields and Their Applications, \textbf{42} (2016), 187--199.

\bibitem{sol1}
L. A. Vinh, \textit{On the solvability of systems of bilinear equations in finite fields}, Proceedings of the American Mathematical Society, \textbf{137}(9) (2009), 2889–2898.

\bibitem{sol2}
L. A. Vinh, \textit{The solvability of norm, bilinear and quadratic equations over finite fields via spectra of graphs}, Forum mathematicum, \textbf{26}(1) (2014), 141–175.



\end{thebibliography}
\end{document}